\documentclass[12pt]{amsart}
\usepackage{amssymb,amsmath,amsthm,mathrsfs,MnSymbol}
\usepackage[dvipsnames]{xcolor}
\usepackage{esint}
\usepackage{graphicx}
\usepackage[linktocpage=true,backref=page]{hyperref}
\hypersetup{
	colorlinks=true, 	
	linktoc=all,     		
	linkcolor=RoyalPurple,  	
	citecolor=Emerald,
}
\calclayout
\newcommand\scalemath[2]{\scalebox{#1}{\mbox{\ensuremath{\displaystyle #2}}}}

\newtheorem{theorem}{Theorem}[section]
\numberwithin{theorem}{section} \numberwithin{equation}{section}
\newtheorem{corollary}[theorem]{Corollary}
\newtheorem{proposition}[theorem]{Proposition}
\newtheorem{definition}[theorem]{Definition}
\newtheorem{lemma}[theorem]{Lemma}

\theoremstyle{remark}
\newtheorem{remark}[theorem]{Remark}
\newtheorem{example}[theorem]{Example}

\newcommand{\pd}[2][]{\frac{\partial #1}{\partial #2}}
\title[Line congruences associated to Appell functions]{Line congruences associated to Appell's hypergeometric functions of rank-4}
\author{Matthew Ryan}
\address{Dept.\!~of Mathematics, Virginia Tech, Blacksburg, VA 24060}
\email{mattryan810@vt.edu}
\author{Michael T.  Schultz}
\address{Dept.\!~of Mathematics, Virginia Tech, Blacksburg, VA 24060}
\email{michaelschultz@vt.edu}
\begin{document}
\begin{abstract}
Line congruences are the genesis of important examples of transformations of projective surfaces, such as the Laplace transform. We survey and review results related to this historical subject, then derive original formulae for the Laplace transform of the entire rank-4 linear system associated to such an immersed projective surface. We apply our results to study the geometry of surfaces defined by Appell's hypergeometric functions of rank-4: namely, $F_2$ and $F_4$. We show that the sequence of Laplace invariants for each is determined respectively by the Euler-Poisson-Darboux equation for $F_2$, and Darboux's Harmonic equation for $F_4$. Further, we show the natural line congruences generated by the Laplace transforms of each constitute a $W$-congruence, an important example of line congruence in which a surface and its Laplace transform are simultaneously locally conformally equivalent.
\end{abstract}
\maketitle
%
%
\section{Introduction}
\label{s-intro}
Our primary objects of study in this article are from classical differential geometry: projective surfaces $S \subset \mathbb{P}^3$, as well as line congruences, which by definition is a two-parameter family $\mathbf{L}$ of lines in $\mathbb{P}^3$. A line congruence and can be usefully thought of as a pair of surfaces $\mathbf{L} = \{S,S^\prime\}$ such that each line intersects the surfaces in a suitable way. We investigate in particular, local differential geometric structures on $S$ that come from immersing $S$ into projective space, and how such geometry on a pair of surfaces $S,S^\prime$ impacts $\mathbf{L}$. This subject is quite old and was historically important in influencing both the development of projective and differential geometry of surfaces. As we will see, there are intimate and inevitable connections to linear PDEs as well, which is part of why historically there was much attention payed to developing the subject. As such, many great mathematicians over the generations have contributed to it. Moreover, projective surfaces feature prominently in algebraic geometry, and so this area is a fertile mix of many beautiful areas of mathematics.
\par  As a line congruence is essentially a pair of surfaces, we spend some time to talk about surfaces in projective space. Since we are interested in issues related to differential geometry and differential equations, we consider projective space defined over a complete field $\mathbb{K}$, which is either $\mathbb{R}$ or $\mathbb{C}$. In algebraic geometry, a (connected) projective surface $S \subset \mathbb{P}^3$ is defined by the locus of solutions to a single (irreducible) homogeneous polynomial equation of degree $n$ in the variables $X_0,X_1,X_2,X_3$, where $\mathbb{P}^3 = \mathbb{P}(X_0,X_1,X_2,X_3)$. For example, a quadric surface is just the image in $\mathbb{P}^3$ of the vanishing locus of quadratic form on the 4-dimensional vector space $\mathbb{K}^4$. When $\mathbb{K} = \mathbb{R}$, we can refer to a quadric surface by the signature of the quadratic form. An important example are quadrics of signature $(2,2)$, which include the surface $Q$ defined by the equation
\begin{equation}
	\label{eq-quadric}
	Q \; : \quad X_0X_3 - X_1X_2 = 0 \, .
\end{equation}
In fact, every signature $(2,2)$ quadric is projectively equivalent to (\ref{eq-quadric}), while for $\mathbb{K}=\mathbb{C}$ every quadric is projectively equivalent to (\ref{eq-quadric}). 
\par The algebraic conditions described above are global properties of a projective surface $S \subset \mathbb{P}^3$. In fact, it is possible to detect aspects of these geometric qualities from the local projective differential geometry of $S$, which in turn can be determined from knowledge of an important system of linear PDEs whose solutions define an immersion of the surface into $\mathbb{P}^3$. In \S \ref{s-proj_DG}, we review the pertinent aspects of the connections between projective differential geometry and linear systems of PDE of rank-4.
\par This prepares us to discuss line congruences $\mathbf{L}= \{S,S^\prime\}$ in \S \ref{s-line_congruences}. We will be particularly interested in when $S,S^\prime$ are mutually tangent to $\mathbf{L}$ in a suitable way; such surfaces are known as \emph{focal} surfaces of a congruence. The Laplace transform provides a mechanism of transforming a given surface $S$ to a new surface $S^\prime$ such that both are focal surfaces for a certain line congruence $\mathbf{L} = \{S,S^\prime\}$. A Weingarten congruence, or $W$-congruence, is a line congruence in which the local conformal structures of the projective second fundamental forms on $S$ and $S^\prime$ are simultaneously equivalent; historically, much attention has been paid to these congruences. 
\par Our main results in \S \ref{s-line_congruences} (Theorems \ref{thm-positive_laplace_system}, \ref{thm-negative_laplace_system}) characterize the rank-4 linear systems of Laplace transformed surfaces purely in terms of the rank-4 linear system of the original surface. We believe that this is the first time the full formulae have appeared in the literature in this form. These results can be shown to be in perfect agreement with some known results in the literature \cite[\S 3.2]{MR2216951}, but their novelty is that our formulae are expressible solely in terms of the rank-4 linear system (\ref{eq-rank_4_system}) of the original surface equipped with so-called \emph{conjugate coordinates}, a system of coordinates on a projective surface that will be introduced in \S \ref{ss-focal}. These coordinates are an important notion and so we will provide three characterizations of these in Definitions \ref{def-conjugate_coords1}, \ref{def-conjugate_coords2}, and \ref{def-conjugate_coords3}. We calculate the fundamental invariants of the projective metric and cubic form which determine fundamental aspects of the geometry of $S$ (see Proposition \ref{prop-vanishing_invariants}) in conjugate coordinates in Theorem \ref{thm-conjugate_invariants}. Moreover, we provide explicit formulae in Corollaries \ref{cor-positive_W}, \ref{cor-negative_W} for the Weingarten invariant $W$, a quantity whose nonvanishing provides the only obstruction to a line congruence $\mathbf{L}$ being a $W$-congruence, purely in terms of the rank-4 linear system. To the best of our knowledge, this is the first time these quantities have appeared in the literature in this form.
\par In the remaining sections we apply the results of \S \ref{s-line_congruences} to surfaces that are determined from certain special functions, namely Appell's bivariate hypergeometric functions $F_2$ and $F_4$. The surfaces $S_{F_2}$ and $S_{F_4}$ that they define in $\mathbb{P}^3$ are special, since they inherently depend on a number of parameters that appear in the definition of the hypergeometric functions. For example, for certain values of the parameters (in fact, for certain linear subvarieties in the space of parameters) these surfaces become quadric surfaces, projectively equivalent to the surface $Q$ in (\ref{eq-quadric}). Generally, we can change the parameters and not alter the fundamental geometry of the surfaces in significant ways; such surfaces are called \emph{projectively applicable}. The surfaces $S_{F_2}$ and $S_{F_4}$ were studied by Sasaki \cite{MR1858701} in a method distinct from our approach in this paper.
\par We give background information on Appell's hypergeometric functions in \S \ref{ss-integral_reps}, giving their well-known integral representations. From here, in \S \ref{ss-GKZ}, we show how the GKZ method \cite{MR1080980} can be implemented to produce the rank-4 linear systems that annihilate $F_2$ and $F_4$. Though the GKZ method is well-known, perhaps particularly to certain types of algebraic geometers, we believe it will be useful to newcomers to the method to have a complete example of it being implemented in the literature. Finally, in \S \ref{s-hypergeometric_congruences}, we apply our computational results of \S \ref{s-line_congruences} to $S_{F_2}$ and $S_{F_4}$. The Laplace transforms and associated line congruences of such surfaces has not been studied much before. We show that the sequence of Laplace invariants for $F_2$ is equivalent to the Euler-Poisson-Darboux equation (\ref{eq-EPD}), and the sequence for $F_4$ is equivalent to Darboux's Harmonic equation (\ref{eq-harmonic_equation}). Moreover, we show in Theorems \ref{thm-F2_W_congruence} and \ref{thm-F4_W_congruence} that the Laplace transforms of $S_{F_2}$ and $S_{F_4}$ are $W$-congruences.
\par It is interesting to note that the entire genesis of this article comes from a certain beautiful special function identity that the second author learned a number of years ago due to Clingher, Doran, \& Malmendier \cite[Corollary 2.2]{MR3767270}, that relates Appell's $F_2$ function to the Gauss hypergeometric function $_2F_1$ as a two-parameter Euler integral transform:
\begin{equation*}
	\scalemath{.8}{ F_2\left(\left.\begin{array}{c}
			\alpha ; \beta_1, \beta_2 \\
			\gamma_1, \gamma_2
		\end{array} \right\rvert\, \frac{1}{s}, 1-\frac{t}{s}\right)=-\frac{\Gamma\left(\gamma_2\right)s^\alpha(s-t)^{1-\gamma_2}}{\Gamma\left(\beta_2\right) \Gamma\left(\gamma_2-\beta_2\right)}
		  \int_s^t \frac{d u}{(s-u)^{1-\beta_2}(u-t)^{1+\beta_2-\gamma_2} u^\alpha}{ } \, _2 F_1\left(\left.\begin{array}{c}
			\alpha, \beta_1 \\
			\gamma_1
		\end{array} \right\rvert\, \frac{1}{u}\right) .}
\end{equation*}
Specifics will be recalled in Lemma \ref{lem-F2_integral_transform}. This article can be thought of as an extended interpretation of and an attempt to deeply understand this identity.
\subsection{Conventions \& Preliminaries}
\label{ss-prelims}
\par In the scope of this article, our results are applicable in both the real $C^\infty$ category and complex analytic category, though it is necessary to impose a mild assumption in the $C^\infty$ case that will be described in \S \ref{s-line_congruences}. We let $\mathbb{K}$ denote either $\mathbb{R}$ or $\mathbb{C}$ as we consider working in either respective category. We let $V^n \cong \mathbb{K}^n$ denote an $n$-dimensional vector space over $\mathbb{K}$, and we let $\bigwedge^k V^n$ be the vector space of antisymmetric $k$-forms on $V^n$, contained in the exterior algebra $\bigwedge^\bullet V^n = \bigoplus_{k=0}^n \bigwedge^k V^n $. We will work with projective spaces $\mathbb{P}^n$ and Grassmannians $\mathrm{Gr}(m,V^{n+1})$, whose points respectively parameterize lines and $m$-planes in the vector space $V^{n+1}$ and are each respectively smooth, compact manifolds of dimension $n$ and $m(n+1-m)$. By definition, $\mathbb{P}^n = V^{n+1} - \{0\} / \sim$, where $v \sim v^\prime$ is the equivalence relation defined by $v^\prime = kv$ for some nonzero $k \in \mathbb{K}$. A similar, but more involved definition yields $\mathrm{Gr}(m,V^{n+1})$. The canonical projection map 
\begin{equation}
	\label{eq-projective_quotient}
		\pi : \, V^{n+1} - \{0\} \to \mathbb{P}^n \, , \quad \quad v \mapsto [v]
\end{equation}
induced by equivalence relation $\sim$ allows for the use of the familiar homogeneous coordinates on $\mathbb{P}^n$ and lifts to $V^{n+1}$, which we freely use. A nice summary of this introductory material in \cite[Part 1, \S 1.1 - 1.4]{ward1990twistor}. 
\par Of course, every projective space is a Grassmannian, $\mathbb{P}^n = \mathrm{Gr}(1,V^{n+1})$. The first Grassmannian that is not a projective space is $\mathrm{Gr}(2,V^4)$, which as said above naturally appears in questions related to the geometry of $\mathbb{P}^3$ as the points in $\mathrm{Gr}(2,V^4)$ parameterize lines in $\mathbb{P}^3$. Given two points $p_1,p_2 \in \mathbb{P}^3$, we can associate the line $\overline{p_1p_2} \cong \mathbb{P}^1 \subset \mathbb{P}^3$ passing through $p_1$ and $p_2$ with the equivalence class of the plane $[p_1 \wedge p_2] \in \mathbb{P}(\bigwedge^2 V^4) = \mathbb{P}^5$. Concretely, this defines a smooth (algebraic, in fact) map $\rho : \mathrm{Gr}(2,V^4) \to \mathbb{P}^5$ called the \emph{Pl\"ucker embedding}, defined by the formula
\begin{equation}
	\label{eq-plucker_embedding}
	\rho : \overline{p_1p_2} \mapsto [p_{01} \, : \, p_{02} \, : \, p_{03} \, : \, p_{12} \, : \, p_{13} \, : \, p_{23} \,] \, , \quad p_{ij} := X^1_iX^2_j - X^1_jX^2_i \, 
\end{equation}
where $p_k = [\, X^i_0 \, : \, \cdots \, : \, X^i_3 \,]$, $k=1,2$. The coordinates $p_{ij}$ yielding the image of the Pl\"ucker embedding satisfy the quadratic equation
\begin{equation}
	\label{eq-klein_quadric}
	Q^4 \; : \quad p_{01}p_{23} - p_{02}p_{13} + p_{03}p_{23} = 0 \, ,
\end{equation}
defining a 4-dimensional quadric hypersurface $Q^4 \subset \mathbb{P}^5$ called the Klein quadric. The Pl\"ucker embedding is an isomorphism, so $\mathrm{Gr}(2,4) \cong Q^4$ as projective algebraic varieties. We will naturally encounter subvarieties (that is, algebraic submanifolds) of $Q^4$ from our geometric considerations in $\mathbb{P}^3$.
\par For a manifold $M$ over $\mathbb{K}$, all relevant geometric data is assumed to be smooth in the appropriate category even if it is not said explicitly. Over $\mathbb{C}$ we take smooth to mean holomorphic and a diffeomorphism to be a biholomorphism, and we always work in the analytic topology. If $U$ is an open set, we take $C^\infty(U)$ to denote the local algebra of smooth or holomorphic functions if $\mathbb{K}=\mathbb{R},\mathbb{C}$, respectively. We use familiar notation for common Lie groups, taken over the appropriate $\mathbb{K}$ as necessary. Similarly, we use familiar notation for bundles of tensors and differential forms on a manifold $M$. Hence $T^r_s(M)$ is the bundle of $(r,s)$ tensors on $M$ with $r$ covariant and $s$ contravariant indices, while $S^k(M), \, \Omega^k(M) \subset T^k_0(M)$ and denotes respectively the bundles of symmetric and antisymmetric covariant forms. By misuse of notation, we write $t \in T^r_s(M)$ to denote a section of the tensor bundle. When local coordinates have been equipped, we use the usual summation convention on repeated indices, unless otherwise noted.
\par The following section provides a quick review of Laplace invariants of second order hyperbolic PDE in the plane, which plays an important role in this article.
\subsubsection{Laplace Invariants}
\label{sss-laplace_invariants}
\par In this arena one naturally encounters linear, second order, scalar hyperbolic PDE in the plane of the form
\begin{equation}
	\label{eq-hyperbolic_PDE}
	z_{xy} + a z_x + b z_y + c z = 0 \, ,
\end{equation}
where $a,b,c$ are smooth functions in an open set $(x,y) \in U$ and subscripts as usual denote partial derivatives. These PDE have been studied extensively and historically, see for example Darboux \cite{MR1324110} for a classical account, as well as Iwasaki \cite{MR946650} or Zeitsch \cite{math6120316} for their role in more modern considerations. Equation (\ref{eq-hyperbolic_PDE}) arises naturally in connection to the geometry of projective surfaces $S \subset \mathbb{P}^3$, where the variables $(x,y)$ are certain local coordinates on $S$ called \emph{conjugate coordinates} that will be defined in ~\S \ref{ss-focal}.
\par It is important to understand geometric aspects related to (\ref{eq-hyperbolic_PDE}). The equation itself is not uniquely determined. One seeks so-called \emph{covariant} transformations that preserve the form of (\ref{eq-hyperbolic_PDE}). Due to linearity and homogeneity, we may rescale the unknown $z \mapsto  fz$ for $f \in C^\infty(U)$, under which (\ref{eq-hyperbolic_PDE}) transforms as
\begin{equation}
	\label{eq-transform1}
	z_{xy} + \hat{a} z_x + \hat{b} z_y + \hat{c} z = 0 \, ,
\end{equation}
where the components $\hat{a},\hat{b},\hat{c}$ are given by
\begin{equation}
	\label{eq-transform1_components}
	\hat{a} = a+(\log f)_y \, , \quad
		\hat{b} = b+(\log f)_x \, , \quad
		\hat{c} = c+a(\log f)_x+b(\log f)_y+f_{x y} / f \, .
\end{equation}
A transformation $z \mapsto fz$ will be called a projective gauge transformation of (\ref{eq-hyperbolic_PDE}); we say the two equations (\ref{eq-hyperbolic_PDE}) and (\ref{eq-transform1}) are \emph{projectively equivalent}. Moreover, one may change coordinates $(x,y) \mapsto (u=u(x),v=v(y))$, which transforms (\ref{eq-hyperbolic_PDE}) as 
\begin{equation}
	\label{eq-transform2}
\small	z_{uv}+a \frac{d y}{d v} z_u+b \frac{d x}{d u} z_v+c \frac{d x}{d u} \frac{d y}{d v} z=0 \,  .
\end{equation}
In addition, one may consider the involutive transformation of swapping the role of $x$ and $y$. These transformations have a natural geometric interpretation in terms of the geometry of a projective surface equipped with conjugate coordinates. Such matters will be discussed in detail in \S \ref{s-proj_DG} and \S \ref{s-line_congruences}.
\par To find geometric data encoded by (\ref{eq-hyperbolic_PDE}), one seeks quantities that are invariant under the transformations above. To that end, one finds that the quantities 
\begin{equation}
	\label{eq-laplace_invariants}
	\scalemath{1}{\mathsf{h}=a b+a_x-c, \quad \mathsf{k}=a b+b_y-c }
\end{equation}
have natural invariance properties. If we set $\hat{\mathsf{h}},\hat{\mathsf{k}}$ by the analogous quantities for (\ref{eq-transform1}) and (\ref{eq-transform2}), then one finds $\hat{\mathsf{h}} =\mathsf{h} $ and $\hat{\mathsf{k}} = \mathsf{k}$ from (\ref{eq-transform1}), and there is equality of the symmetric 2-forms
\begin{equation}
	\label{eq-laplace_2-forms}
	\scalemath{.85}{ \mathsf{h} (dx \otimes dy + dy \otimes dx) = \hat{\mathsf{h}}(du \otimes dv + dv \otimes du) \, , \quad
		\mathsf{k} (dx \otimes dy + dy \otimes dx) = \hat{\mathsf{k}}(du \otimes dv + dv \otimes du) \, }
\end{equation}
from (\ref{eq-transform2}). One makes the tensors in (\ref{eq-laplace_2-forms}) symmetric to take into account the involutive symmetry of exchanging coordinates. 
\begin{definition}
	\label{def-laplace_invariants}
	The quantities $\mathsf{h}$ and $\mathsf{k}$ (or the associated symmetric 2-forms) are called the \emph{Laplace invariants} of (\ref{eq-hyperbolic_PDE}). As a convention, we will refer to $\mathsf{h}$ as the \emph{positive} Laplace invariant and $\mathsf{k}$ as the \emph{negative} Laplace invariant.
\end{definition}
\par In general, the solution space of (\ref{eq-hyperbolic_PDE}) is infinite dimensional. However, sometimes there are ``fundamental'' solutions that come from either an ODE or a system of (coupled) PDEs with a finite dimensional solution space. Such a fundamental solution $R(x,y)$ is called a \emph{Riemann} function of (\ref{eq-hyperbolic_PDE}). See \cite[\S 1]{math6120316} for a complete definition, which is not strictly necessary for our purposes.
\par It is thus of interest to determine to what extent (\ref{eq-hyperbolic_PDE}) is integrable. Laplace invariants provide an insight into this question. For example, the condition $h = 0$ implies that (\ref{eq-hyperbolic_PDE}) is solvable explicitly by quadrature, see \cite[Rmk. 4.4]{MR2216951}. There are other conditions by which the equation simplifies. For example, when the positive and negative Laplace invariants are equal $\mathsf{h} = \mathsf{k}$, we have that $a_x = b_y$. Hence the $1$-form $\phi = bdx + ady$ is exact, so by the transformation property (\ref{eq-transform1_components}), the equation can be simplified to $z_{xy} + \hat{c}z = 0$ for a function $f$ satisfying $-d(\log f) = \phi$.
\par We recall two such hyperbolic equations.
\begin{example}
	\label{ex-EPD}
	\begin{flushleft}
		\textbf{Euler-Poisson-Darboux Equation.} This equation is of the form
		\begin{equation}
			\label{eq-EPD}
			z_{xy} - \frac{\beta^\prime}{x-y}z_x - \frac{\beta}{x-y}z_y = 0 \, ,
		\end{equation}
		where $\beta,\beta^\prime \in \mathbb{K}$ are constants. The equation has been well studied, as its name would suggest. See for example \cite{MR323959}, \cite[\S 4.6]{MR2216951} for a modern perspective. For $\beta^\prime = -\frac{1}{2}$, $\beta=\frac{1}{2}$, the second author found that these equations also arise naturally in connection with the geometry of K3 surfaces \cite[Cor. 6.2.86, Prop. 6.2.92]{schultz_geometry_2021}. The Laplace invariants of the EPD equation are
		\begin{equation}
			\label{eq-EPD_laplace_invariants}
		\mathsf{h} = \frac{\beta(\beta^\prime+1)}{(x-y)^2} \; , \quad \mathsf{k} = \frac{\beta^\prime(\beta+1)}{(x-y)^2} \; .
		\end{equation}
	\end{flushleft}
\end{example}
\begin{example}
	\label{ex-harmonic_equation}
	\begin{flushleft}
		\textbf{Harmonic Equation.} This equation is of the form 
		\begin{equation}
			\label{eq-harmonic_equation}
		z_{xy} + \left(\frac{\alpha(\alpha-1)}{(x-y)^2}-\frac{\beta(\beta-1)}{(x+y)^2}\right)z = 0 \, ,
		\end{equation}
		for constants $\alpha,\beta$, and was named and studied by Darboux \cite{MR1324110}. A modern perspective relevant to our aims here is provided by Iwasaki \cite{MR946650}, who showed that a Riemann function can be given in terms of Appell's $F_4$ function. Since the coefficients $a = 0 = b$ in (\ref{eq-hyperbolic_PDE}), the Laplace invariants $\mathsf{h} = \mathsf{k} = -c$ are equal,
		\begin{equation}
		\label{eq-harmonic_laplace_invariants}
			\mathsf{h} = \mathsf{k} = -\left(\frac{\alpha(\alpha-1)}{(x-y)^2}-\frac{\beta(\beta-1)}{(x+y)^2}\right) \, .
		\end{equation}
	\end{flushleft}
\end{example}
\section{Projective differential geometry of surfaces}
\label{s-proj_DG}
\par In what follows, we utilize the presentation of relevant results in \cite[\S 2]{MR2216951} and \cite[\S 3]{ashleyschultz2025}. More detailed descriptions can be found there. Let $S \subset \mathbb{P}^3$ be a surface equipped with local coordinates $(x,y)$ in coordinate neighborhood $U$. We assume that $S$ is both connected and orientable. We will say that two surfaces $S \cong S^\prime$ are isomorphic if there is an automorphism $g \in \mathrm{PGL}(4)$ such that $S^\prime = g \cdot S$. Shrinking the local coordinate neighborhood $U$ if necessary, we assume that $\psi$ is an isomorphism, and identify $S \cong \psi(U)$.
\subsection{Geometry of rank-4 linear systems}
\label{ss-rank-4_systems}
\par It was known to the old masters (see e.g., Wilczynski \cite{MR1500552,MR1500783,MR0131232}, Darboux \cite{MR1324110}, and references therein) that the data $(S,\psi)$ is determined by a coupled system of linear PDE of rank-4 in the local coordinates $(x,y)$ of the form 
\begin{equation}
	\label{eq-rank_4_system}
	\begin{cases}
		z_{xx} = \ell z_{xy} + a z_x + b z_y + p z \, , \\
		z_{yy} = m z_{xy} + c z_x + f z_y + q z \, , \\
	\end{cases}
\end{equation}
where the coefficients are smooth functions of the local coordinates and subscripts indicate partial derivatives. That (\ref{eq-rank_4_system}) is of rank-4 means that there are local solutions $z_0(x,y),z_1(x,y),z_2(x,y),z_3(x,y)$ that determine the immersion $\psi$ by the formula
\begin{equation}
	\label{eq-immersion}
	\psi(x,y) = [z_0(x,y) : z_1(x,y) : z_2(x,y) : z_3(x,y)] \, .
\end{equation}
Given a system of the form (\ref{eq-rank_4_system}) of rank-4, it is clear that (\ref{eq-immersion}) defines a surface $S \subset \mathbb{P}^3$. Conversely, given $(S,\psi)$, the components of $\psi$ given by (\ref{eq-immersion}) necessarily satisfy a system of the form (\ref{eq-rank_4_system}) of rank-4. Such a system is of rank-4 if and only if the coefficient functions satisfy a set of nonlinear differential constraints known as \emph{integrability conditions}, which can be obtained by examining the equation $(z_{xx})_{yy} = (z_{yy})_{xx}$.
\par In fact, the coefficient functions $\ell,m,\dots,p,q$ themselves determine local projective differential geometric data of $S \subset \mathbb{P}^3$. Wilczynski \cite{MR1500783} determined that the most general form of transformation of the independent and dependent variables in the linear system (\ref{eq-rank_4_system}) that preserves its form is given by $(x,y;z) \mapsto (u,v;f(u,v)z)$, where $x = x(u,v), y= y(u,v)$ defines a local diffeomorphism $\phi : S \to S$ and $f : S \to \mathbb{K}$ is a nonvanishing smooth function.
\par One differential geometric manifestation of the system is gained by writing the system in Pfaffian form 
\begin{equation}
	\label{eq-connection_form}
	d e = \omega e \, ,
\end{equation}
where $e = (\psi,\psi_x,\psi_y,\psi_{xy})$ and $\omega$ is a $4 \times 4$ matrix-valued 1-form whose entries are obtained in a straightforward manner from (\ref{eq-rank_4_system}) called the \emph{connection form}. Equation (\ref{eq-connection_form}) is called the \emph{moving frame equation}. The connection form $\omega$ determines a connection $\nabla$ on a rank-4 vector bundle $E \cong L \oplus TS \oplus L^*,$ where $L = \mathrm{span}(\psi)$, $TS = \mathrm{span}(\psi_x,\psi_y)$ is the tangent bundle of $S$, and $L^* = \mathrm{span}(\psi_{xy})$. Then $L^* \cong N_S$ is identified with the normal bundle of $S \subset \mathbb{P}^3$, where $N_S := T\mathbb{P}^3 |_S / TS$. A gauge group $G \subset \mathrm{PGL}(4)$ is determined by the transformations that stabilize this decomposition. Then $G$ acts by $e \mapsto \tilde{e}=ge$ and $\omega \mapsto \tilde{\omega} = g \omega g^{-1} + dg \cdot g^{-1}$
, where $g \in G$ is a smooth $G$-valued function on $S$ \cite[Eq. (3.6)]{ashleyschultz2025}. It is well-known that $\omega$ satisfies the relation
\begin{equation}
	\label{eq-integrability_conditions}
	d\omega - \omega \wedge \omega = 0 \, ,
\end{equation}
which is known as the \emph{Mauer-Cartan equation}. Equation (\ref{eq-integrability_conditions}) simply says that the connection $\nabla$ is flat, i.e. that the curvature of $\nabla$ is zero. Moreover, the integrability conditions $(z_{xx})_{yy} = (z_{yy})_{xx}$ are equivalent to
the Mauer-Cartan equation (\ref{eq-integrability_conditions}).
\par From (\ref{eq-rank_4_system}), $\omega$ is easily computed as 
\begin{equation}
	\label{eq-rank_4_connection_form}
	\scalemath{.85}{\omega=\left(\begin{array}{cccc}
		0 & d x & d y & 0 \\
		p \, d x & a \, d x & b \, d x & \ell \, d x+d y \\
		q \, d y & c \, d y & f \, d y & d x+m \, d y \\
		* & * & * & *
	\end{array}\right) \, ,}
\end{equation}
allowing for the explicit computation of projective differential geometric quantities. Write the frame $e = (e_0,e_1,e_2,e_3)$, where we regard each as row vectors.  When using indices on the frame $e$, we signify tangent directions by letting latin characters range from $1,2$, and general indices $0,\dots,3$ denoted by greek indices. By the definition of $e_3 = \psi_{xy}$, we have $d e_0=\omega_0^0 e_0+\omega_0^1 e_1+\omega_0^2 e_2,$ so that $\omega^3_0 = 0$. Define $\omega^i := \omega^i_0$ for $i=1,2$; from (\ref{eq-rank_4_connection_form}) these are nothing but $dx$ and $dy$. The integrability conditions (\ref{eq-integrability_conditions}) reveal that by exterior differentiation of the equation $\omega^3_0 = 0$ that $\omega^i \wedge \omega^3_i = 0$. From the classical Cartan lemma on the exterior algebra \cite[Thm. 4.4]{Sternberg1964}, there are smooth functions $h_{ij} = h_{ji}$ such that $\omega^3_i = h_{ij}\omega^i$. Thus it is natural to consider the normal bundle valued symmetric covariant tensor $h \in S^2(S) \otimes N_S$ given by $	h := h_{ij} \, \omega^i \otimes \omega^j$. We always assume that $h$ is nondegenerate. The form $h$ plays the role of the projective second fundamental form of $S \subset \mathbb{P}^3$. From (\ref{eq-rank_4_connection_form}), the form $h$ is computed simply as 
\begin{equation}
	\label{eq-second_fundamental_form}
	h = \ell \, d x \otimes d x+d x \otimes d y+d y \otimes d x+m \, d y \otimes d y \, .
\end{equation}
It is commonplace to make the additional assumption that in the real case $\mathbb{K} = \mathbb{R}$ that $h$ has signature $(1,1)$, which we shall also freely assume in this article. Such surfaces are the projective analogue of hyperbolic surfaces in euclidean differential geometry. It is straightforward to show, that under gauge transformation, the conformal class $[h]$ is preserved, so we are free to consider $h$ up to scaling by a nonvanishing function. It was shown in \cite[Prop. 3.3]{ashleyschultz2025} that the covariant transformations of Wilczynski $z \mapsto fz$ act isometrically on $h$. 
\par There is also a symmetric cubic tensor $\Phi$ that plays a significant role. It is convenient to transform the frame $e \mapsto \tilde{e} = ge$ by suitable gauge transformation $g \in G$ so that $|\det(h)|=1$ and $\omega^0_0 + \omega^3_3 = 0$. Then dropping decorations on components of the transformed objects $\tilde{\omega}$ and $\tilde{h}$, define $\Phi_{ijk}$ by the equation
\begin{equation}
	\label{eq-cubic_components}
	\Phi_{i j k} \omega^k = d h_{i j}-h_{k j} \omega_j^k-h_{i k} \omega_j^k .
\end{equation}
The $\Phi_{ijk}$ are then totally symmetric in the indices $i,j,k = 1,2$. It is natural then to consider the symmetric cubic form $	\Phi=\Phi_{ijk} \, \omega^i \otimes \omega^j \otimes \omega^k$. An analogous computation as above shows that the conformal class $[\Phi]$ is independent of the choice of projective frame $e$. 
\par It turns out to be useful to pick a specific representative $\varphi \in [h]$ that is an absolute projective invariant. To this end, one defines a scalar $\mathcal{F}$ by contracting $\Phi^{\otimes 2}$ with the inverse $h^{-1} = (h^{ij})$ as
\begin{equation}
	\label{eq-fubini_scalar}
	\mathcal{F}=\Phi_{i j k} \Phi_{p q r} h^{i p} h^{j q} h^{k r} 
\end{equation}
called the Fubini scalar. The transformation properties of $\mathcal{F}$ under a change of frame imply that $\varphi := \mathcal{F}h \in [h]$ is an absolute invariant.
\begin{definition}
	\label{def-projective_metric}
	The tensor $\varphi = \mathcal{F}h \in [h]$ is called the \emph{projective metric} of $S \subset \mathbb{P}^3$.
\end{definition}
Although $h$ is always assumed to be nondegenerate, the projective metric $\varphi$ may vanish, even identically, according to the behavior of the Fubini scalar $\mathcal{F}$. This leads to significant restrictions on the geometry of $S \subset \mathbb{P}^3$, as will be explained shortly. However, a careful examination of the integrability conditions $d\omega - \omega \wedge \omega = 0$ reveals that it is possible to deform the pair $(S,\psi)$ while keeping both $\varphi$ and $\Phi$ the same. Such is the case when the linear system (\ref{eq-rank_4_system}) contains auxiliary parameters or is allowed to depend on arbitrary functions, neither of which uniquely determine $S$. A surface admitting such a deformation is called \emph{projectively applicable}; see Definition \ref{def-proj_applicable} below. These notions are relevant to our discussion in \S \ref{s-appell}, \ref{s-hypergeometric_congruences}, when we discuss surfaces related to Appell's bivariate hypergeometric functions which ultimately depend on a small number of parameters.
\subsection{Digression on Quadrics}
\label{ss-quadrics_proj_applicable}
\par As a concrete example, consider the case of a linear system that corresponds to the quadric surface $Q \subset \mathbb{P}^3$ defined by (\ref{eq-quadric}). An important aspect of the geometry of $Q$ is that it forms a so-called \emph{ruled} surface. 
\begin{definition}
	\label{def-ruled_surface}
	A projective surface $R \subset \mathbb{P}^3$ is said to be \emph{ruled} if $R$ contains a 1-parameter family of projective lines $\mathbb{P}_x^1$, where $x$ is the local parameter of a curve in $\mathbb{P}^3$. Equivalently, a ruled surface $R \subset \mathbb{P}^3$ is given by the data of a curve in the Grassmannian $\mathrm{Gr}(2,4)$ or equivalently on the Klein quadric $Q^4 \subset \mathbb{P}^5$ in (\ref{eq-klein_quadric}).
\end{definition}
\par A practical way of representing $R$ is by giving a pair of distinct curves $s(x),t(x)$ that each lie on $R$ with common local parameter $x$, and writing $R = \{s(x),t(x)\}$ to denote the union of all lines $\mathbb{P}_x^1$ joining the points $s(x)$, $t(x)$. Then the ruled surface $R$ determines a curve in the Grassmannian $\mathrm{Gr}(2,4)$ (or equivalently on the Klein quadric $Q^4$) from the correspondence $x \mapsto [s(x) \wedge t(x)] \in \mathrm{Gr}(2,4)$ and then composing with the Pl\"ucker embedding (\ref{eq-plucker_embedding}) to get the curve on $Q^4$. Ruled surfaces are the projective analogue of the familiar ruled surfaces in euclidean geometry, such as the cylinder, cone, helicoid, and hyperbolic paraboloid. Of these, the hyperbolic paraboloid is an example of a quadric surface in euclidean space, and is well known to admit two independent rulings. The projective version of the hyperbolic paraboloid is precisely the quadric surface $Q$ in (\ref{eq-quadric}), which admits the independent rulings $x \mapsto [1 \, : \, x \, : \, 0 \, :\, 0]$, $y \mapsto [0 \, : \, 0 \, : \, 1 \, :\, y]$ for $x,y$ affine coordinates on independent copies of $\mathbb{P}^1$. Over $\mathbb{K} = \mathbb{C}$ all quadrics are doubly ruled. Conversely, any doubly ruled surface is projectively equivalent to a quadric. When we refer to a ruled surface $R$, we will always assume that $R$ is only singly ruled; more discussion appears in \S \ref{ss-ruled_surfaces_congruences}.
\par It is easy to see that the independent rulings given above determine an immersion $\psi : Q \to \mathbb{P}^3$ as 
\begin{equation}
	\label{eq-quadric_immersion}
	\psi(x,y) = [\, 1 \, : \, x \, : \, y \, : \, xy \,]
\end{equation}
whose respective components $z_0,\dots,z_3$ clearly satisfy the defining equation (\ref{eq-quadric}) of $Q$. By the considerations above, these components must be realizable as linearly independent solutions of a rank-4 linear system (\ref{eq-rank_4_system}). It is easy to see that such a system is nothing but
\begin{equation}
	\label{eq-quadric_system}
		z_{xx} = 0 \, , \quad
		z_{yy} = 0 \, .
\end{equation}
\par This has a number of consequences on the projective differential geometry of $Q$. First, by examining (\ref{eq-second_fundamental_form}), $h$ is simply the constant tensor $h = dx \otimes dy + dy \otimes dx$. However, this is not unique to quadric surfaces, but is a geometric manifestation of the restriction that $h$ have signature $(1,1)$. Any coordinate system on a general projective surface $S \subset \mathbb{P}^3$ in which $\ell = 0 = m$ is called an \emph{asymptotic coordinate system} on $S$. Asymptotic coordinates allows for study of $S$ in terms of the family of so-called \emph{Lie quadrics} attached to $S$; at each point $p \in S$, the Lie quadric $Q_p \cong Q$ is the quadric surface projectively equivalent to (\ref{eq-quadric}) that has second order contact with $S$ at $p$. In concrete terms, this means that locally the coordinate lines $x = \mathrm{const.}$, $y = \mathrm{const.}$ agree up to second order derivatives with the two independent ruling lines on the quadric $Q_p$. Given our considerations, such a coordinate system always exists on any projective surface $S \subset \mathbb{P}^3$, see e.g. \cite[\S 3.3]{ashleyschultz2025}. 
\par Thus the ruling lines $(x,y)$ on the quadric surface (\ref{eq-quadric}) are themselves asymptotic coordinates. More significantly, one finds by computation of the cubic form that $\Phi = 0$ for a quadric surface. Then trivially  $\mathcal{F} = 0$, so the projective metric $\varphi = 0$ for a $Q$. In fact, the vanishing $\Phi = 0$ is also is sufficient for a projective surface $S \subset \mathbb{P}^3$ to be (contained in) a quadric surface $Q$, see See \cite[Thm. 2.3]{MR2216951}. 
\par It is natural to ask upon what conditions the general rank-4 linear system (\ref{eq-rank_4_system}) maps to a quadric. Assume that $S$ is equipped with asymptotic coordinates $(x,y)$. It is always possible, by a suitable projective gauge transformation $z \mapsto fz$, to reduce the rank-4 linear system (\ref{eq-rank_4_system}) to 
\begin{equation}
	\label{eq-canonical_system}
	\begin{cases}
		z_{xx} = b z_y + p z \, , \\
		z_{yy} = c z_x + q z \, ,
	\end{cases}
\end{equation}
called the \emph{canonical system}. The canonical system (\ref{eq-canonical_system}) plays an important role in determining aspects of the geometry of $S$ related to the associated family of Lie quadrics and has historically been carefully studied. For example, it allows for a simple expression of the cubic form $\Phi$ and projective metric $\varphi$, since in this case $\{dx,dy\}$ is already a suitable coframe to compute in \cite[Eq. (2.15)]{MR2216951}, \cite[\S 2]{MR1762804}:
\begin{equation}
	\label{eq-invariants_asymptotic}
	\Phi =-2b dx^{\otimes 3} -2c dy^{\otimes 3} \, ,\quad
		\varphi = 8bc(dx \otimes dy + dy \otimes dx) \,.
\end{equation}
Thus $S \cong Q \subset \mathbb{P}^3$ is a quadric if and only if $\Phi = 0$ if and only if $b = c = 0$. 
\par In the case of a quadric $Q$, the integrability conditions reduce to $p_y = 0$ and $q_x = 0$. Hence the canonical system (\ref{eq-canonical_system}) reduces to a decoupled system of second order ODEs, as 
\begin{equation}
	\label{eq-ODEs}
		z_{xx} = p(x) z \, , \quad
		z_{yy} = q(y) z \, ,
\end{equation}
where $p(x)$ and $q(y)$ are otherwise unrestrained. The elementary theory of linear ODEs asserts that there are independent solutions $u_1(x),u_2(x)$ and $w_1(y),w_2(y)$ of each row in (\ref{eq-ODEs}), and thus an immersion $\psi : Q \to \mathbb{P}^3$ is given by
\begin{equation}
	\label{eq-general_quadric_immersion}
	\psi(x,y) = [\, u_1(x)w_1(y) \, : \, u_1(x)w_2(y) \, : \, u_2(x)w_1(y) \, : \, u_2(x)w_2(y) \,] \, .
\end{equation}
This clearly lies on the quadric surface (\ref{eq-quadric}). Hence the rank-4 system mapping into a quadric surface depends on two arbitrary functions $p(x)$ and $q(y)$, which can be viewed as ``deformations'' of the immersion $(Q,\psi)$ of the quadric surface. Sasaki \& Yoshida \cite{MR960834} have studied this in detail. We have the following terminology.
\begin{definition}
	\label{def-proj_applicable}
	An immersed surface $\psi : S \to \mathbb{P}^3$ admitting a non-trivial deformation is called \emph{projectively applicable}. An immersed surface admitting no such deformation is called \emph{very general}. 
\end{definition}
\par The precise definition of a projective deformation is not necessary for our purposes, and for example can be found in \cite[Def. 3.7]{ashleyschultz2025}. A heuristic way of detecting the existence of a projective deformation of $(S,\psi)$ is from the presence of arbitrary functions or parameters that do not change $\{\varphi,\Phi\}$. These types of systems will arise in \S \ref{s-appell} and \S \ref{s-hypergeometric_congruences} when we discuss Appell's hypergeometric functions. The relative vanishing of $\varphi$ and $\Phi$ can be summarized by the following result.
\begin{proposition}
	\label{prop-vanishing_invariants}
	For an immersed surface $S \subset \mathbb{P}^3$:
	\begin{enumerate}
		\item[(\textit{i})] For $S$ very general, $\varphi$ and $\Phi$ are differential invariants of $S$, that is, they locally characterize $S$ up to projective motion. 
		\item[(\textit{ii})] $S \cong Q$ is a quadric surface if and only if $\Phi=0$ vanishes. Consequently, the projective metric $\varphi = 0$ also vanishes for quadric surfaces.
		\item[(\textit{iii})] $S \cong R$ is a ruled surface if and only if $\varphi=0$ vanishes and $\Phi \neq 0$.
		\item[(\textit{iv})] if $S$ is projectively applicable with both $\varphi \neq 0$ and $\Phi \neq 0$, then the rank-4 linear system (\ref{eq-rank_4_system}) depends on at most three independent parameters.
	\end{enumerate}
\end{proposition}
\section{Line congruences associated to projective surfaces}
\label{s-line_congruences}
\par In the previous sections we have shown that a certain representative $\varphi \in [h]$ of the conformal class of the projective second fundamental form of $S \subset \mathbb{P}^3$ plays an important role in determining the projective differential geometry of the surface. Our assumptions on $h$ however, as is often the case in $2$-dimensional conformal geometry, imply from the existence of asymptotic coordinates on $S$ that $h$ is locally conformally flat. This can be interpreted in terms of how the conformal structure of $S$ locally compares near $p \in S$ with the conformal structure of the Lie quadric $Q_p$. A natural question is that if $S_1$ and $S_2$ are two distinct yet ``closely related'' surfaces (in a yet undefined way, to be introduced in \S \ref{ss-focal}), if it is possible to make a simultaneous comparison of their conformal structures in a geometrically meaningful way. 
\par Here we will describe precisely such a mechanism that will allow one to do so, the so-called \emph{Laplace transform}. This is a classical operation that has received significant attention for its usefulness in both the theory of projective surfaces and its connection to integrability of hyperbolic linear PDE in the plane, see e.g. \cite[\S 4]{MR2216951} and the references therein. In order to do so, we first must introduce some terminology related to ruled surfaces and line congruences, which are key ingredients in the geometry we wish to study.
\subsection{Ruled Surfaces and Line Congruences}
\label{ss-ruled_surfaces_congruences}
Let $R \subset \mathbb{P}^3$ be a ruled surface as in Definition \ref{def-ruled_surface}. If the line $\mathbb{P}_x^1$ is affinely parameterized by $y$, an immersion $\psi : R \to \mathbb{P}^3$ is given by $\psi(x,y) = s(x) + yt(x) \, \in \mathbb{P}^3$. It is straightforward to show that the parameters $(x,y)$ give asymptotic coordinates on $S$, so that each component $z_k(x,y)$ of $\psi$ is a solution of the canonical system (\ref{eq-canonical_system}) of $S$. That each component be linear in $y$ simply means that $c=0=q$. This is enough to confirm Proposition \ref{prop-vanishing_invariants} (iii). A similar analysis of the integrability conditions implies that the canonical system of a ruled surface can be expressed as 
\begin{equation}
	\label{eq-ruled_surface_canonical_system}
	z_{xx}=\left(\alpha(x) y^2+\beta(x) y+\gamma(x)\right) z_y+(-\alpha(x) y+\delta(x)) z \, , \quad z_{yy}=0
\end{equation}
for some choice of functions $\alpha(x),\beta(x),\gamma(x),\delta(x)$. This is sufficient to see that a ruled surface is projectively applicable.
\par There are several important examples of ruled surfaces beyond the doubly ruled quadric that play an important role in the context of this article. Given a curve $s(x)$ in $\mathbb{P}^3$, define the ruled surface $	R = \{s(x),t(x) = v + ks(x)\},$
where $k \in \mathbb{K}$ and $v \in V^4 - \{0\}$ are fixed. This ruled surface is called a \emph{cone} over the curve $s(x)$. From the curve $s(x)$ we can obtain another ruled surface as $R = \{s(x),t(x) = s^\prime(x)\}$, which gives a ruled surface consisting of tangent lines of the curve $s(x)$. This surface is called a \emph{tangent developable surface}. 
\par There is a notion of ruled surface that includes both cones and tangent developable surfaces. 
\begin{definition}
	\label{def-developable}
	A ruled surface $R = \{s(x),t(x)\}$ is said to be \emph{developable} if 
	\begin{equation}
		\label{eq-developable}
		s(x) \wedge t(x) \wedge s^\prime(x) \wedge t^\prime(x) = 0 \in \Omega^4(V^4) \cong C^\infty(V^4) \, .
	\end{equation}
	Hence a developable ruled surface is locally either a cone or a tangent developable surface. In either case, the generating curve $s(x)$ is called the \emph{directrix curve} of the developable surface.
\end{definition}
Developable surfaces will be play an important role in what follows. 
\begin{definition}
	\label{def-line_congruence}
	A line congruence $\mathbf{L}$ is a two-parameter family of lines in $\mathbb{P}^3$. Equivalently, a line congruence is a surface in the Grassmannian $\mathrm{Gr}(2,4)$, or equivalently on the Klein quadric $Q^4 \subset \mathbb{P}^5$ in (\ref{eq-klein_quadric}).
\end{definition}
\par A line congruence is quite a general object in $\mathbb{P}^3$. One way of constructing a line congruence $\mathbf{L}$ is as follows. Given a projective surface $S \subset \mathbb{P}^3$, first find a family of curves $\mathcal{C}$ that locally foliate $S$. That is to say every curve in the family $\mathcal{C}$ lies on $S$, and each point $p \in S$ lies on only a single curve $C \in \mathcal{C}$, at least in some neighborhood of the point $p$. Then a line congruence $\mathbf{L}$ can be constructed as
\begin{equation}
	\label{eq-tangent_congruence}
	\mathbf{L} := \bigcup_{p \in S} \overline{T_p C} \, ,
\end{equation} 
where $\overline{T_p C} \cong \mathbb{P}^1$ is the projective compactification of the tangent line $T_p C \subset T_pS$. The line congruence $\mathbf{L}$ is then naturally parameterized by the surface $S$. This is an important object so we formally define it.
\begin{definition}
	\label{def-tangent_congruence}
	The line congruence $\mathbf{L}$ in (\ref{eq-tangent_congruence}) parameterized by the surface $S$ is called the \emph{tangent congruence} associated to the family of curves $\mathcal{C}$ that lie on $S$.
\end{definition}
In an analogous way to the presentation of ruled surfaces above, a practical way of presenting a line congruence $\mathbf{L}$ is by giving a pair of surfaces $S_1,S_2$ such that each line in $\mathbf{L}$ intersects both $S_1$ and $S_2$ in a suitable way, and write $\mathbf{L} = \{S_1,S_2\}$. To be more precise, suppose that for $i=1,2$, the immersed surface $(S_i,\psi_i)$ both share a common local parameter space $(x,y) \in U$. Then the line congruence $\mathbf{L}$ can be denoted as
\begin{equation}
	\label{eq-local_congruence}
	\mathbf{L} = \{\psi_1(x,y),\psi_2(x,y)\}
\end{equation}
to signify the union of all lines $\mathbb{P}_{x,y}^1$ that joins the respective points $\psi_1(x,y) \in S_1$ and $\psi_2(x,y) \in S_2$. Then $\mathbf{L}$ determines a surface in the Grassmannian $\mathrm{Gr}(2,4)$ (or equivalently on the Klein quadric $Q^4$) from the correspondence $	(x,y) \mapsto [\psi_1(x,y) \wedge \psi_2(x,y)] \in \mathrm{Gr}(2,4)$ and then composing with the Pl\"ucker embedding (\ref{eq-plucker_embedding}) to get the surface on $Q^4$.
\par  Of course, from a given $\mathbf{L}$, there are many surfaces $S_1,S_2 \subset \mathbb{P}^3$ such that $\mathbf{L} = \{S_1,S_2\}$. It is of interest to determine then, from a given $\mathbf{L}$, when there is a representation in terms of surfaces $S_1,S_2$ as in (\ref{eq-local_congruence}) that has some geometric significance for the surfaces other than merely intersecting them suitably. For example, one may ask whether $\mathbf{L}$ is the tangent congruence (\ref{eq-tangent_congruence}) for some family of curves $\mathcal{C}$ on one of the surfaces $S_i$, or even if $\mathbf{L}$ is a simultaneous tangent congruence for some families $\mathcal{C}_i$ of curves on the respective surfaces $S_i$. Such a surface is called a \emph{focal} surface of the congruence $\mathbf{L}$.
\par This can be made precise via Definition \ref{def-developable}. Let $\mathbf{L} = \{\psi_1,\psi_2\}$ and $c(t) = (x(t),y(t))$ be a curve in the parameter space $U$. Then there is a corresponding ruled surface $R_c = \{\psi_1(c(t)),\psi_2(c(t))\} \subset \mathbf{L}$, and per Equation (\ref{eq-developable}), $R_c$ is developable when
\begin{equation}
	\label{eq-immersion_ruled_surfaces}
	\psi_1 \wedge \psi_2 \wedge \psi_1^\prime \wedge \psi_2^\prime = 0 \, .
\end{equation}
It follows that at each point in $u \in U$, there are generically two directions $X_i \in T_uU$, $i=1,2$, such that along the integral curves $c_i(t)$, the ruled surfaces $R_{c_i}$ are developable; moreover, the respective families $\mathcal{C}_i$ of directrix curves for each locally foliate surfaces $S_i$. It is reasonable to insist that these directions $X_1,X_2$ are mutually compatible, in the sense that the Lie derivative $[X_1,X_2]=0$ vanishes.
\begin{definition}
	\label{def-focal_surfaces}
	The surfaces $S_1,S_2$ composed of families of directrix curves $\mathcal{C}_1,\mathcal{C}_2$ as above are called \emph{focal surfaces} of the line congruence $\mathbf{L}$.
\end{definition}
Notice trivially that the focal surfaces generate the line congruence $\mathbf{L} = \{S_1,S_2\}$. When the ruled surfaces above are all tangent developable surfaces, then $\mathbf{L}$ is simultaneously a tangent congruence for each $(S_i,\mathcal{C}_i)$, where $\mathcal{C}_i$ foliates $S_i$ by the directrix curves.
\subsection{Projective surfaces as focal surfaces of line congruences}
\label{ss-focal}
In this section we examine an inverse problem that follows from the considerations above: when are a pair of surfaces $(S_1,S_2)$ focal surfaces of a line congruence $\mathbf{L}$? This is a classical problem that has been studied in a variety of contexts by many mathematicians of previous generations, see again \cite{MR2216951} and the many references throughout. We are interested in the case that $\mathbf{L} = \{S_1,S_2\}$ is simultaneously a tangent congruence of both surfaces $S_1$ and $S_2$. 
\subsubsection{Conjugate Coordinates}
\label{sss-conjugate_coordinates}
\par To that end, let us study (\ref{eq-immersion_ruled_surfaces}) in more detail. To ease notation, let us identify $\psi_1 \equiv z$ and $\psi_2 \equiv w$ to denote the unknowns in the corresponding rank-4 linear systems (\ref{eq-rank_4_system}). From the chain rule, (\ref{eq-immersion_ruled_surfaces}) can be written as the vanishing of quadratic form
\begin{equation}
	\label{eq-developable_quad_form}
	P(x^\prime)^2 + 2Qx^\prime y^\prime + R(y^\prime)^2 = 0 \, ,
\end{equation}
where $P,Q,R$ are determined from $z$ and $w$ as
\begin{equation}
	\label{eq-quad_coeffs}
	\scalemath{.7}{\begin{cases}
		P=z \wedge w \wedge z_x \wedge w_x \, , \\
		2 Q=z \wedge w \wedge z_x \wedge w_y+z \wedge w \wedge z_y \wedge w_x \, , \\
		R=z \wedge w \wedge z_y \wedge w_y \, .
	\end{cases}}
\end{equation}
The straightening lemma (see e.g., \cite[Prop. 1.53]{warner1971foundations}) assures that that since we are on a $2$-dimensional parameter space $U$, we may find new local coordinates $(x,y)$ in $U$ that straighten the vector fields $X_1 = \partial_x$ and $X_2 = \partial_y$ along whose integral curves the corresponding ruled surfaces are developable. In the new coordinate system, these integral curves of the vector fields $X_1,X_2$ are nothing but the coordinate lines $x = \mathrm{const.}, \; y= \mathrm{const.}$, respectively. Then $X_1 = \partial_x$ and $X_2 = \partial_y$ are the directions in which (\ref{eq-developable_quad_form}) is satisfied. 
\begin{definition}
	\label{def-conjugate_coords1}
	The coordinates $(x,y)$ described above are called \emph{conjugate coordinates} on the surfaces $S_1,S_2$. 
\end{definition}
\par An intrinsic description of conjugate coordinates on a surface $S \subset \mathbb{P}^3$ will be given shortly. First let us consider the implication of (\ref{eq-quad_coeffs}) in the conjugate coordinate system $(x,y)$. In the direction $X_2 = \partial_y$, that is along a coordinate line $x = \mathrm{const.}$, (\ref{eq-developable_quad_form}) reduces to $R=0$. Assuming that $Q \neq 0$, this implies that there are functions $\alpha,\beta \in C^\infty(U)$ such that
\begin{equation}
	\label{eq-R=0_consequence}
	w = \alpha z_y + \beta z \, .
\end{equation}
This equation says that at each point of $S_2$, $w$ lies along a line tangent to the surface $S_1$ in the direction of $\partial_y$. On the other hand, along the coordinate line $y = \mathrm{const.}$, we have $P=0$ and so we get that
\begin{equation}
	\label{eq-P=0_consequence}
	z = \gamma w_x + \delta w
\end{equation}
for some $\gamma,\delta \in C^\infty(U)$. Analogously as above, this equation implies that each point of $S_1$, $z$ lies along a line tangent to the surface $S_2$ in the direction of $\partial_x$. This is precisely what is meant by the surfaces $S_1$ and $S_2$ simultaneously being focal surfaces for the line congruence $\mathbf{L}$. 
\par Evaluating (\ref{eq-P=0_consequence}) with (\ref{eq-R=0_consequence}) and eliminating $w$, a simple calculation shows that 
\begin{equation}
	\label{eq-conjugate_eq1}
	z_{xy}+a z_x+b z_y+c z = 0 \, 
\end{equation}
where $a,b,c \in C^\infty(U)$ are simple functions of $\alpha,\beta,\gamma,\delta$ and relevant derivatives. This is precisely the type of second order hyperbolic PDE that was discussed in \S \ref{sss-laplace_invariants}. The assumption that $Q \neq 0$ implies that $z,z_x,z_y,z_{xx}$ are linearly independent across $U$. Thus there is linear relationship
\begin{equation}
	\label{eq-conjugate_eq2}
	z_{yy} + qz_{xx} + m z_x +n z_y + r z = 0 \, ,
\end{equation}
where the coefficient functions can be determined again from (\ref{eq-R=0_consequence}) and (\ref{eq-P=0_consequence})\footnote{Note that the appearance of the characters $a,b,c,q$ etc. is distinct from the role that they played in (\ref{eq-rank_4_system}).}. But these equations are then nothing but the rank-4 linear system (\ref{eq-rank_4_system}) expressed in conjugate coordinates with respect to the frame $e = (z,z_x,z_y,z_{xx})$. It is clear that analogous statements hold for $w$ as well. This gives us an intrinsic definition of conjugate coordinates on a surface $S$ in terms of the rank-4 linear system defining an immersion $\psi : S \to \mathbb{P}^3$. 
\begin{definition}
	\label{def-conjugate_coords2}
	A coordinate system $(x,y)$ on a surface $S \subset \mathbb{P}^3$ is called a system of \emph{conjugate coordinates} if the rank-4 linear system defining an immersion $\psi : S \to \mathbb{P}^3$ can be expressed as
	\begin{equation}
		\label{eq-rank-4_system_conjugate}
		\scalemath{1}{\begin{cases}
			z_{xy}+a z_x+b z_y+c z = 0 \, , \\
			z_{yy} + qz_{xx} + m z_x +n z_y + r z = 0 \, , 
		\end{cases}}
	\end{equation}
	that is, one of the equations in the coupled linear system is a second order hyperbolic equation as in Equation (\ref{eq-hyperbolic_PDE}).
\end{definition}
\par Then the discussion of \S \ref{ss-rank-4_systems} carries through in conjugate coordinates with respect to the frame $e = (z,z_x,z_y,z_{xx})$; the moving frame equation is again $de = \omega e$ where $\omega$ is analogously computed as in (\ref{eq-rank_4_connection_form}), and the integrability conditions are determined by the Maurer-Cartan equation $d\omega - \omega \wedge \omega =0$. In particular, we find the conformal class $[h]$ of the projective second fundamental form is generated by the diagonal tensor 
\begin{equation}
	\label{eq-conjugate_second_fundamental_form}
	h = dx \otimes dx - q \, dy \otimes dy \, .
\end{equation}
This gives our final characterization of conjugate coordinates. 
\begin{definition}
	\label{def-conjugate_coords3}
	A system of coordinates $(x,y)$ on a projective surface $S \subset \mathbb{P}^3$ are called \emph{conjugate coordinates} if the projective second fundamental form $h \in S^2(S) \otimes N_S$ is a diagonal tensor as in Equation (\ref{eq-conjugate_second_fundamental_form}).
\end{definition}
The terminology \emph{conjugate} comes from the following. Given a non-degenerate symmetric 2-form $g \in S^2(S)$, one says that two vector fields $X_1,X_2 \in TS$ are conjugate relative to $g$ if 
\begin{equation*}
	g(X_1,X_2) = 0 \, .
\end{equation*}
Provided that $[X_1,X_2]=0$, one may straighten the vector fields as above; in the resulting coordinate system $g$ will be a diagonal tensor. Clearly conjugate directions are independent of the conformal class $[g]$, which leads us back to Definition \ref{def-conjugate_coords3}. 
\par We end this section by providing expressions for the cubic form $\Phi$ and projective metric $\varphi$ in conjugate coordinates, analogous to the quantities that appear in (\ref{eq-invariants_asymptotic}) that were expressed in asymptotic coordinates.
\begin{theorem}
	\label{thm-conjugate_invariants}
	Let $S \subset \mathbb{P}^3$ be an immersed surface equipped with conjugate coordinates $(x,y)$, determined by the rank-4 linear system in Equation (\ref{eq-rank-4_system_conjugate}). Define the quantities $A,B$ by
	\begin{equation}
		\label{eq-AB_components}
			A = q_x + 4bq - 2m \, , \quad
			B = -q_y + (4a - 2n)q \, .
	\end{equation}
	Then there is a coframe $\{\omega^1,\omega^2\}$ of $T^*S$ in which the symmetric cubic form (\ref{eq-cubic_components}) $\Phi = \Phi_{ijk} \omega^i \otimes \omega^j \otimes \omega^k$ has components given by
	\begin{equation}
		\label{eq-cubic_components_conjugate}
		\Phi_{111} = -A \, , \quad
			\Phi_{112} = B \, , \quad
			\Phi_{122} = -qA \, ,\quad
			\Phi_{222} = qB \, ,
	\end{equation}
	and the Fubini scalar $\mathcal{F}$ (\ref{eq-fubini_scalar}) is given by 
	\begin{equation}
		\label{eq-fubini_scalar_conjugate}
		\scalemath{.8}{\mathcal{F} = \frac{8\sqrt{q}}{5}\left(qA^2-B^2\right) \, .}
	\end{equation}
	Thus the projective metric $\varphi$ is given by $\varphi = \mathcal{F}h$, which in the coframe $\{\omega^1,\omega^2\}$ takes the form
	\begin{equation}
		\scalemath{.8}{\label{eq-projective_metric_conjugate}
		\varphi = \mathcal{F}\left(\sqrt{q} \,\omega^1 \otimes \omega^1 - \frac{1}{\sqrt{q}} \, \omega^2 \otimes \omega^2\right) \, .}
	\end{equation}
\end{theorem}
\begin{proof}
	After computing the connection form $\omega$ for the rank-4 linear system (\ref{eq-rank-4_system_conjugate}) from the moving frame equation $de = \omega e$ with respect to the frame $e = (z,z_x,z_y,z_{xx})$, the computation of the above quantities follows from the discussion in \S \ref{ss-rank-4_systems} and Sasaki \& Yoshida \cite[\S 4.2]{MR960834}, which also produces the coframe $\{\omega^1,\omega^2\}\subset T^*S$.
\end{proof}
\begin{corollary}
	\label{cor-quadric_conjugate}
	An immersed surface $S \subset \mathbb{P}^3$ equipped with conjugate coordinates $(x,y)$ and determined from the rank-4 linear system (\ref{eq-rank-4_system_conjugate}) is isomorphic to a quadric surface $S \cong Q$ if and only if $A = 0 = B$, where $A,B$ are given in (\ref{eq-AB_components}). The surface $S \cong R$ is a ruled surface if and only if $\mathcal{F} = 0$ and $A,B \neq 0$.
\end{corollary}
\begin{proof}
	This follows from the expressions in Equations (\ref{eq-cubic_components_conjugate}) and (\ref{eq-fubini_scalar_conjugate}) and Proposition \ref{prop-vanishing_invariants} (\textit{ii}), (\textit{iii}).
\end{proof}
\subsubsection{The Laplace Transformation}
\label{sss-laplace_transformation}
Let $S \subset \mathbb{P}^3$ be an immersed surface corresponding to $z$ as in the previous section, equipped with conjugate coordinates $(x,y)$ and satisfying the rank-4 linear system appearing in (\ref{eq-rank-4_system_conjugate}). Consider the corresponding tangent congruence $\mathbf{L}$ consisting of tangent lines to the family $\mathcal{C}$ of coordinate lines on $S$, of which $S$ is naturally a focal surface. To find another immersed focal surface $S^\prime$ such that $\mathbf{L} = \{S,S^\prime\}$, first choose a point $w$ along the tangent line in the direction of $\partial_y$ as in (\ref{eq-R=0_consequence}). Due to homogeneous coordinates, $w$ can be written as 
\begin{equation*}
	w = z_y + \lambda z
\end{equation*}
for some function $\lambda$. Calculating $w_x$ and using the relation (\ref{eq-conjugate_eq1}), we have that 
\begin{equation*}
	w_x=(\lambda-a) z_x-b z_y+(\lambda_x-c) z \, ,
\end{equation*}
we find that (\ref{eq-P=0_consequence}) is satisfied precisely when $\lambda = a$. We have the following.
\begin{definition}
	\label{def-laplace_transformation}
	The quantity $w = z_y + az$ is called the \emph{first Laplace transform} of $z$, and is written as $w=z^+$. Similarly, the quantity $z^- = z_x + bz$ is called the \emph{minus first Laplace transform}. By convention, we will call $z^\pm$ respectively the \emph{positive} and \emph{negative} Laplace transforms of $z$. 
\end{definition}
\par An easy computation reveals the following connection to the Laplace invariants $\mathsf{h},\mathsf{k}$ defined in Equation (\ref{eq-laplace_invariants}).
\begin{proposition}
	\label{prop-laplace_composition}
	Composing the positive and negative Laplace transforms yields the following: $(z^+)^- = \mathsf{h} z$, $(z^-)^+ = \mathsf{k}z$. Hence if $\mathsf{h} \neq 0$, $(z^+)^-$ is projectively equivalent to $z$ and analogously for $\mathsf{k} \neq 0$. If $\mathsf{h} = \mathsf{k} \neq 0$, then the positive and negative Laplace transforms commute.
\end{proposition}

\par If a Laplace invariant $\mathsf{h},\mathsf{k} = 0$ vanishes, one says that the corresponding Laplace transform is \emph{degenerate}. In this case the corresponding $z^\pm$ will end up defining a curve rather than a surface. Necessary and sufficient that this be the case is when the original surface is a tangent developable surface as in Definition \ref{def-developable}, see \cite[\S 4.2]{MR2216951}.
\par Clearly, the notion of Laplace transforms can be defined completely in reference to a single hyperbolic PDE as in (\ref{eq-hyperbolic_PDE}); indeed, there has been substantial inquiry into utilizing Laplace transforms in questions of integrability in this realm alone. However we may also think of the Laplace transform as a geometric transformation on an immersed surface. To study the resulting geometry, we pursue the following.
\subsection{Laplace transform of a rank-4 linear system}
\label{ss-laplace_transform_rank-4}
\par Starting from an immersed surface $S \subset \mathbb{P}^3$, to demand that a Laplace transform $z^\pm$ define a surface $S^\pm \subset \mathbb{P}^3$, it is necessary and sufficient that $z^\pm$ satisfy a rank-4 linear system of the general form in Definition \ref{def-conjugate_coords2}. Then we will have $\mathbf{L}^\pm = \{S,S^\pm\}$ with both surfaces focal, and the parameters $(x,y)$ will be simultaneously conjugate coordinates for both. In this case, we will refer to the resulting surface $S^\pm$ as the (positive / negative) Laplace transform of the original surface $S$. In the next section, we will derive the corresponding rank-4 linear systems of $S^\pm$ in relation to the original rank-4 linear system (\ref{eq-rank-4_system_conjugate}) for $S$.
\par We consider the case of $z^+ = z_y + az$, the negative Laplace transform $z^-$ is handled similarly and we will state the corresponding results after. To ease notation we will set $w = z^+$ and $u = z^-$. Utilizing the hyperbolic equation (\ref{eq-conjugate_eq1}), we find that 
\begin{equation}
	\label{eq-diff_wx}
	w_x = -bw + \mathsf{h}z \, .
\end{equation}
It follows from Proposition \ref{prop-laplace_composition} that it is useful to consider the case $\mathsf{h} \neq 0$. Differentiating again and utilizing both $w = z_y + az$ and (\ref{eq-diff_wx}), one finds the hyperbolic equation
\begin{equation}
	\label{eq-positive_laplace_transform1}
	w_{xy} + a_1 w_x + b_1 w_y + c_1 w = 0 \, ,
\end{equation}
where 
\begin{equation}
	\label{eq-laplace_transform_components1}
		a_1=a-(\log \mathsf{h})_y \, , \quad
		b_1=b \,, \quad
		c_1=c-a_x+b_y-b(\log \mathsf{h})_y \, .
\end{equation}
Computing the second equation in the coupled system (\ref{eq-rank-4_system_conjugate}) is a bit more involved but along the same lines; in this case one must also use the so-called \emph{prolonged} version of the hyperbolic equation (\ref{eq-conjugate_eq1})
\begin{equation}
	\label{eq-prolonged_hyperbolic}
z_{x x y}=\left(b c-c_x\right) z+\left(b a-a_x-c\right)z_x+\left(b^2-b_x\right)z_y-az_{x x} \, ,
\end{equation}
which amounts to differentiating (\ref{eq-conjugate_eq1}) and back substituting the original coupled system to expressing everything in terms of the frame $e = (z,z_x,z_y,z_{xx})$. Note that it is possible to express $z_{xxy}$ in terms of the frame $e$ since the linear system (\ref{eq-rank-4_system_conjugate}) is integrable of rank-4, and hence is said to be \emph{of finite type}. This ultimately allows one to eliminate all $z$ derivatives when finding the $w$-version of (\ref{eq-conjugate_eq2}). After some work, one finds the following. 
\begin{theorem}
	\label{thm-positive_laplace_system}
	Assume that $\mathsf{h} \neq 0$. Then the positive Laplace transform $w = z^+$ defines an immersed surface $S^+ \subset \mathbb{P}^3$ and satisfies a rank-4 linear system
	\begin{equation}
		\label{eq-positive_rank-4_system}
		\begin{cases}
			w_{xy}+a_1 w_x+b_1 w_y+c_1 w = 0 \, , \\
			w_{yy} + q_1w_{xx} + m_1 w_x +n_1 w_y + r_1 w = 0 \, , 
		\end{cases}
	\end{equation}
	where $a_1,b_1,c_1$ are given in Equation (\ref{eq-laplace_transform_components1}) and 
	\begin{equation}
		\label{eq-positive_q}
		q_1 = \frac{\left(-2 a_x+\mathsf{h}\right) q}{\mathsf{h}}+\frac{m_y}{\mathsf{h}}-\frac{mq_y}{q \mathsf{h}} \, .
	\end{equation}
	The remaining components $m_1,n_1,r_1$ are given in (\ref{eq-pos_components1}, \ref{eq-pos_components2}) in Appendix \ref{appendix}.
\end{theorem}
Notice in particular that Theorem \ref{thm-positive_laplace_system} allows for the computation of the conformal structure $[h^+]$ on the surface $S^+$ defined by $w = z^+$ as being generated by the diagonal tensor $h^+ = dx \otimes dx - q_1 dy \otimes dy,$
where $q_1$ is given in (\ref{eq-positive_q}). Hence we have 
\begin{equation}
	\label{eq-positive_conformal_structure}
	[h^+] = \left[\mathsf{h}q dx \otimes dx - (-2q^2a_x+qm_y-mq_y+q^2\mathsf{h})dy \otimes dy\right] \, .
\end{equation}
This allows for a simultaneous comparison of the conformal structures $[h]$ on $S$ and $[h^+]$ on $S^+$.
\begin{definition}
	\label{def-W_congruence}
	Suppose that $\mathbf{L} = \{S,S^\prime\}$ with both $S,S^\prime$ focal surfaces corresponding to $z$ and $w$, parameterized locally by the same conjugate coordinate system $(x,y)$. Suppose that the conformal structures $[h_z]$ and $[h_w]$ are generated by the diagonal tensors $h_z = \alpha dx \otimes dx + \beta dy \otimes dy$, $h_w = \alpha^\prime dx \otimes dx + \beta^\prime dy \otimes dy$,
	where $\alpha,\beta,\alpha^\prime,\beta^\prime$ are smooth functions. Then the line congruence $\mathbf{L}$ is said to be a \emph{Weingarten congruence}, or simply a \emph{$W$-congruence}, if the conformal structures $[h_z]$ and $[h_w]$ are equivalent. This happens if and only if the \emph{Weingarten invariant} $W$ vanishes,
	\begin{equation}
		\label{eq-weingarten_invariant}
		W := \alpha \beta^\prime - \beta \alpha^\prime = 0 \, .
	\end{equation}
\end{definition}
\par Applying this terminology in the case of the positive Laplace transform, we have the following.
\begin{corollary}
	\label{cor-positive_W}
	The positive line congruence $\mathbf{L}^+ = \{S,S^+\}$ derived from the positive Laplace transform is a $W$-congruence if and only if 
	\begin{equation}
		\label{eq-positive_weingarten}
		\scalemath{1}{W^+ = 2q^2a_x - qm_y + mq_y = 0 \, .}
	\end{equation}
\end{corollary}
\begin{proof}
	From the conformal structures $[h]$ generated by (\ref{eq-conjugate_second_fundamental_form}) and $[h^+]$ in (\ref{eq-positive_conformal_structure}), one finds the Weingarten invariant $W = qW^+$, where $W^+$ is given by (\ref{eq-positive_weingarten}) above.
\end{proof}
\par An analogous story holds for the negative Laplace transform $u = z^-$. In this case, one needs the prolonged equation for $z_{xxx}$, as in (\ref{eq-prolonged_hyperbolic}).
\begin{theorem}
	\label{thm-negative_laplace_system}
	Assume that $\mathsf{k} \neq 0$. Then the negative Laplace transform $u = z^-$ defines an immersed surface $S^- \subset \mathbb{P}^3$ that satisfies a rank-4 linear system
	\begin{equation}
		\label{eq-negative_rank-4_system}
		\scalemath{1}{\begin{cases}
			u_{xy}+a_0 u_x+b_0 u_y+c_0 u = 0 \, , \\
			u_{yy} + q_0u_{xx} + m_0 u_x +n_0 u_y + r_0 u = 0 \, , 
		\end{cases}}
	\end{equation}
	where $a_0,b_0,c_0$ are given by
	\begin{equation}
		\label{eq-laplace_transform_components2}
		\scalemath{1}{	a_0=a \, , \quad
			b_0=b - (\log \mathsf{k})_x \,, \quad
			c_0=c-b_y+a_x-a(\log \mathsf{k})_x} 
	\end{equation}
	and 
	\begin{equation}
		\label{eq-negative_q}
	\scalemath{1}{	q_0 = \frac{\mathsf{k}q}{\mathsf{k}+n_x-2b_y} \, .}
	\end{equation}
	The remaining components $m_0,n_0,r_0$ are given in (\ref{eq-neg_components1}, \ref{eq-neg_components2}) in Appendix \ref{appendix}.
\end{theorem}
\begin{corollary}
	\label{cor-negative_W}
	The negative line congruence $\mathbf{L}^- = \{S,S^-\}$ derived from the negative Laplace transform is a $W$-congruence if and only if 
	\begin{equation}
		\label{eq-negative_weingarten}
		\scalemath{1}{W^- = 2b_y-n_x = 0 \, .}
	\end{equation}
\end{corollary}
\begin{proof}
	It is obvious from Equation (\ref{eq-negative_q}) that so long as $\mathsf{k} \neq 0$, the vanishing of $W^-$ guarantees that $[h^-]$ is conformal to $[h]$.
\end{proof}
\begin{remark}
	\label{rmk-higher_laplace}
	Since the Laplace transforms of the hyperbolic equation (\ref{eq-conjugate_eq1}) are hyperbolic equations of the same form (provided that $\mathsf{h},\mathsf{k} \neq 0$), one may define the \emph{higher} Laplace invariants by computing the Laplace invariants of successive Laplace transforms, by defining $z^0 := z$, and $z^{n \pm 1} := (z^n)^\pm$ for $n \in \mathbb{Z}$. Hence $z^{\pm 1} = z^\pm$. Then the higher Laplace invariants $\mathsf{h}_n,\mathsf{k}_n$ are defined as the Laplace invariants in (\ref{eq-laplace_invariants}) of $z^n$, that is $\mathsf{h}_n=\mathsf{h}(z^n), \mathsf{k}_n=\mathsf{k}(z^n)$. The higher Laplace invariants satisfy a number of useful relations that allow one to compute them recursively. For example, we have
	\begin{equation}
		\label{eq-higher_laplace}
		\scalemath{1}{\mathsf{h}_{n+1}=2 \mathsf{h}_n-\mathsf{k}_n- \left(\log \mathsf{h}_n\right)_{xy} \, ,
		\quad \quad
		\mathsf{k}_{n+1}=\mathsf{h}_n}
	\end{equation}
	for all $n \in \mathbb{Z}$. See for example \cite[Prop. 4.10]{MR2216951} for proof and a number of other relations that $\mathsf{h}_n,\mathsf{k}_n$ satisfy. The higher Laplace invariants have played an important role in issues related to integrability of the hyperbolic equation (\ref{eq-conjugate_eq1}) (when higher invariants vanish), integrability of nonlinear PDE like the $\sinh$-Gordon equation (when the higher invariants are periodic), and determining subtle aspects of the geometry of projective surfaces that would otherwise be invisible without utilizing Laplace transforms and invariants. 
\end{remark} 
\begin{remark}
	\label{rmk-geometry_of_laplace}
	Provided that the higher Laplace invariants are nonvanishing, there is a corresponding sequence $\{S_n\}_{n \in \mathbb{Z}} \subset \mathbb{P}^3$ of immersed surfaces determined from the original surface $S_0 := S$. These Laplace transformed surfaces may have very different geometry from the original surface. Using the formulae (\ref{eq-cubic_components_conjugate}) and (\ref{eq-fubini_scalar_conjugate}) for the cubic form and projective metric, one may determine highly nonlinear differential constraints on the coefficients of the original conjugate system (\ref{eq-rank-4_system_conjugate}) defining a projective surface $S$ such that the surface $S_n$ defines either a ruled or quadric surface. As an example, in Appendix \ref{appendix}, we provide the differential constraints that are both necessary and sufficient for a line congruences $\mathbf{L}^+\pm = \{S,S^\pm\}$ to have both focal surfaces as quadrics. Historically such questions have been important, see e.g. \cite[\S 9.2]{MR2216951}.
\end{remark}
\par In the sequel, we apply the machinery of Laplace transforms and line congruences to well known systems of rank-4, those that are associated to Appell's bivariate hypergeometric functions $F_2$ and $F_4$.
\section{Appell's hypergeometric functions of rank-4}
\label{s-appell}
\par We begin by recalling the Gauss hypergeometric function $_2F_1$. Some nice references for this section and \S \ref{ss-integral_reps} are Kimura \cite{Kimura1973HypergeometricFO} and Yoshida \cite{yoshida1987}. The hypergeometric function $_2F_1$ is defined by the power series for $x \in \mathbb{C}$ 
\begin{equation}
	\label{eq-gauss_2F1}
	\scalemath{.9}{ _2F_1\left(\left.
	\genfrac{}{}{0pt}{}{\alpha ; \beta}{\gamma}
	\right|\,x\right) =
	\sum_{m=0}^{\infty} \frac{(\alpha)_m(\beta)_m}{(\gamma)_m m!} x^m}
\end{equation}
for $\alpha,\beta,\gamma \in \mathbb{C}$ and $(z)_m = z(z+1)\dots (z + m - 1)$ is the rising Pochhammer symbol. This quantity is so named because the series (\ref{eq-gauss_2F1}) is a generalization of the familiar geometric series. This series has been studied intensely since the time of Euler. Even today, hypergeometric functions are ubiquitous in many areas of mathematics and related fields, continuing to new find applications and relevance.  
\par Notice that the Pochhammer symbol may be written in terms of the Gamma function as $(z)_m =
\frac{\Gamma(z+m)}{\Gamma(z)}$ for $z \notin \mathbb{Z}_{\leq 0}$. Then
the series (\ref{eq-gauss_2F1}) may be written as
\begin{equation}
\scalemath{.9}{	\frac{\Gamma(\gamma)}{\Gamma(\alpha) \Gamma(\beta)} \sum_{m=0}^{\infty} \frac{\Gamma(\alpha+m) \Gamma(\beta+m)}{\Gamma(\gamma+m) \Gamma(1+m)} x^m}
\end{equation}
provided that $\alpha, \beta \notin \mathbb{Z}_{\leq 0}$. A simple argument using the ratio test shows that (\ref{eq-gauss_2F1}) is convergent in the disc $|x| < 1$ in $\mathbb{C}$.
\par There are other useful ways to represent $_2F_1$, such as being represented as an integral. It is well known that if $|x| < 1$ and $\mathrm{Re}(\beta) > \mathrm{Re}(\gamma) > 0$, then
\begin{equation}
	\label{eq-gauss_integral_rep}
	\scalemath{.9}{ _2F_1\left(\left.
	\genfrac{}{}{0pt}{}{\alpha ; \beta}{\gamma}
	\right|\,x\right) = 
	\frac{\Gamma(\gamma)}{\Gamma(\beta) \Gamma(\gamma-\beta)} \int_0^1 u^{\beta-1}(1-u)^{\gamma-\beta-1}(1-x u)^{-\alpha} d u \,,}
\end{equation}
where the contour of integration is chosen carefully to reflect the fact that the integrand consists of generically multivalued, singular functions of $u$. This formula, known to Euler, can be shown by using the generalized binomial expansion
\begin{equation}
	\label{eq-generalized_binomial}
	\scalemath{.9}{(1-x u)^{-\alpha} = \sum_{m=0}^{\infty} \frac{(\alpha)_m}{m!} u^m x^m}
\end{equation}
and then evaluating the resulting integral term by term to recover the $_2F_1$ series.
\par Another very important attribute of $_2F_1$ is that it constitutes the fundamental holomorphic solution at $x = 0$ of a linear, second order homogeneous ODE of the form
\begin{equation}
	\label{eq-hypergeometric_ODE}
\scalemath{1}{	x(1-x) \frac{d^2 y}{d x^2}+[\gamma-(\alpha+\beta+1) x] \frac{d y}{ d x}-\alpha \beta y = 0 \, .}
\end{equation}
This ODE is called the \emph{hypergeometric differential equation}, and is just as ubiquitous as $_2F_1$ itself. The equation is the canonical example of a second order Fuchsian\footnote{That is, the equation is linear, homogeneous, and has only regular singular points} equation with precisely three singular points. It was known to Riemann that any such second order Fuchsian ODE  is equivalent to the hypergeometric ODE (\ref{eq-hypergeometric_ODE}) for some  $\alpha,\beta,\gamma$ up to M\"obius transformation.
\par So successful was the use and application of these that mathematicians naturally sought out closely related expressions in two variables; in particular, ones that were not simply the Cauchy product of different copies of $_2F_1$ \cite[\S 7]{Kimura1973HypergeometricFO}. 
\subsection{Appell's bivariate hypergeometric functions $F_2$ and $F_4$}
\label{ss-integral_reps}
Paul Appell's $F_1,F_2,F_3,F_4$ hypergeometric functions generalize the single-variable Gauss hypergeometric function $_2F_1$ to two variables, each complete with a series definition and integral representation. Each of these functions is a solution to some Fuchsian system of linear PDEs. Specifically, $F_2$, $F_3$, and $F_4$ have rank-4 systems, hence the structure discussed in \S \ref{s-proj_DG}, \ref{s-line_congruences}. The function $F_1$ satisfies a Fuchsian system of rank-3, and thus cannot be discussed from the perspective of this article. It can additionally be shown that the systems for $F_2$ and $F_3$ coincide (although the series themselves are not equivalent), up to a change of parameters and variables \cite[p.~62]{yoshida1987}. It then suffices to study only the systems for $F_2$ and $F_4$.  We term the functions $F_2$ and $F_4$ Appell's hypergeometric functions of rank-4.

The functions $F_2,F_3,$ and $F_4$ are defined by the series
\begin{equation}
	\label{eq-appell_functions}
	\scalemath{1}{\begin{cases}
		F_2\left(\left.
		\genfrac{}{}{0pt}{}{\alpha ; \beta_1, \beta_2}{\gamma_1, \gamma_2}
		\right|\, x, y\right)
		= \sum_{m,n = 0}^\infty \frac{(\alpha)_{m+n}(\beta_1)_m(\beta_2)_n}{(\gamma_1)_m(\gamma_2)_n \,m!\,n!} x^m y^n \, ,\\
		F_3\left(\left.
		\genfrac{}{}{0pt}{}{\alpha _1, \alpha_2; \beta_1, \beta_2}{\gamma}
		\right|\, x, y\right)
		= \sum_{m,n = 0}^\infty \frac{(\alpha_1)_{m}(\alpha_2)_n(\beta_1)_m(\beta_2)_n}{(\gamma)_{m+n} \,m!\,n!} x^m y^n \, ,\\
		F_4\left(\left.
		\genfrac{}{}{0pt}{}{\alpha; \beta}{\gamma_1,\gamma_2}
		\right|\, x, y\right)
		= \sum_{m,n = 0}^\infty \frac{(\alpha)_{m+n}(\beta)_{m+n}}{(\gamma_1)_{m}(\gamma_2)_n \,m!\,n!} x^m y^n \, .
	\end{cases}}
\end{equation}

As with the discussion above for $_2F_1$, the parameters of the left side of the equations are all complex numbers such that the members of the top row are not negative integers. It can be shown that each of these series converge in some suitable neighborhood of the origin $(0,0) \in \mathbb{C}^2$. See Kimura \cite[\S 8]{Kimura1973HypergeometricFO}.

From those definitions, one can derive equivalent integral representations (as given in \cite[p.~71]{yoshida1987})
\begin{equation}
	\label{eq-appell_integral_rep}
	\scalemath{.7}{\begin{cases}
		F_2\left(\left.
		\genfrac{}{}{0pt}{}{\alpha ; \beta_1, \beta_2}{\gamma_1, \gamma_2}
		\right|\, x, y\right) = \frac{\Gamma\left(\gamma_1\right) \Gamma\left(\gamma_2\right)}{\Gamma\left(\beta_1\right) \Gamma\left(\beta_2\right) \Gamma\left(\gamma_1-\beta_1\right) \Gamma\left(\gamma_2-\beta_2\right)}\iint_{\Sigma_2}
		\left(1-x u-y v\right)^{-\alpha}
		(1-v)^{\gamma_2 -\beta_2 - 1 } (1-u)^{\gamma_1-\beta_1- 1}
		u^{\beta_1-1}
		v^{\beta_2-1}
		\, du \, dv \; , \\
		F_3\left(\left.
		\genfrac{}{}{0pt}{}{\alpha _1, \alpha_2; \beta_1, \beta_2}{\gamma}
		\right|\, x, y\right) = \frac{\Gamma(\gamma)}{\Gamma(\beta_1)\Gamma(\beta_2)\Gamma(\gamma - \beta_1 - \beta_2)} \iint_{\Sigma_3}
		(1 - u - v)^{\gamma - \beta_1 - \beta_2 - 1}
		(1 - xu)^{-\alpha_1}	(1-yv)^{-\alpha_2}
		u^{\beta_1 - 1}
		v^{\beta_2 - 1}
		\,du\,dv \; , \\
		F_4\left(\left.
		\genfrac{}{}{0pt}{}{\alpha; \beta}{\gamma_1,\gamma_2}
		\right|\, x, y\right) = \frac{(-1)^{\gamma_1-\alpha} \Gamma(\gamma_1) \Gamma\left(\gamma_2\right)}{\Gamma(\alpha-\gamma_1) \Gamma(1-\alpha+\gamma_1) \Gamma\left(\gamma_1+\gamma_2-\alpha -1\right)} \iint_{\Sigma_4}
		(1- xu -yv)^{-\beta}
		(u+v-uv)^{\gamma_1+\gamma_2-\alpha-2}u^{\alpha-\gamma_2} 
		v^{\alpha-\gamma_1}
		\, du\,dv \; ,
	\end{cases}}
\end{equation}
where the regions of integration for each are 
\begin{equation*}
	\scalemath{.8}{\begin{cases}
		\Sigma_2 = \left\{(u,v) \in \mathbb{R}^2 \, \vert \, 0 \leq u,v \leq 1  \right\} \, , \\
		\Sigma_3 = \left\{(u, v) \in \mathbb{R}^2 \, \vert \, u, v, 1-u-v \geq 0 \, \right\} \, , \\
		\Sigma_4 = \left\{(u,v) \in \mathbb{R}^2 \, \vert \, 0 \leq u \leq 1 \, , \, v \leq 0 \, ,\, u+v-uv \geq 0 \, \right\} \, 
	\end{cases}}
\end{equation*}
and have been carefully chosen to reflect the multivalued nature of the integrands.
\subsection{Deriving hypergeometric systems from the GKZ perspective}
\label{ss-GKZ}
The GKZ\footnote{Named after mathematicians Gel'fand, Kapranov, \& Zelevinsky \cite{MR1080980}.} framework for generalized Euler integrals facilitates the derivation of a system of linear differential equations whose solutions include the original integral, in a similar spirit on how the integral representation for $_2F_1$ satisfies the hypergeometric ODE (\ref{eq-hypergeometric_ODE}). Their work \cite{MR1080980} on so-called $\mathcal{A}$-hypergeometric functions yields a massive generalization of the discussions above, and allows us to compute the Fuchsian systems for $F_2$ and $F_4$ purely from the data of the integral representations in \ref{eq-appell_integral_rep}. We define the generalized Euler integral and concisely describe the process of deriving the $\mathcal{A}$-hypergeometric system, using $F_2$ as an example and treating $F_4$ as a result. Our treatment of this process is brief and has many assumptions and assertions. A more rigorous treatment is available in \cite{MR1080980}.
\par Consider an integral of the form
\begin{equation}
	\label{eq-generalized_euler_integral}
	\scalemath{1}{\oint_\Sigma \prod_{i=1}^m P_i(x_1,\dots,x_k)^{\alpha_i}\,x_1^{\beta_1}\cdots x_k^{\beta_k}\,dx_1\cdots dx_k}
\end{equation}
where $P_i(x_1,\dots,x_k) = \sum v_{\omega} x_1^{\omega_{1}} \cdots x_k^{\omega_{k}}$ is a Laurent polynomial in the variables $x_1,\dots,x_k$ and $\omega_{j} \in \mathbb{Z}$ are (fixed) powers. Here we have written the index $\omega = (\omega_1,\dots,\omega_k) \in \mathbb{Z}^k$ on the coefficients $v_{\omega}$, and the powers $\alpha_i, \beta_i\in \mathbb{C}$ of the $P_i$ are complex. To each Laurent polynomial $P_i$ we associate a finite subset $\mathcal{A}_i \subset \mathbb{Z}^k$ that records the powers of each term. The set $\mathbb{C}^{\mathcal{A}_i}$ denotes the vector space of Laurent polynomials with powers determined by $\mathcal{A}_i$, so that each $P_i \in \mathbb{C}^{\mathcal{A}_i}$.
\par This quantity is called a ``generalized Euler integral''. As above, we want to consider (\ref{eq-generalized_euler_integral}) as an analytic function on $\prod \mathbb{C}^{\mathcal{A}_i}$, i.e., as an analytic function of the coefficients $v_\omega$. Since the integrand in is generically singular and multivalued, care must be taken to ensure that the integral is well defined, by choosing a suitable $k$-dimensional cycle $\Sigma \subset \mathbb{C}^k$ over which to integrate. In \cite[\S 2.2]{MR1080980}, it was shown how to pick such a $\Sigma$. Namely, they showed that there is a certain chain complex defined on the complement $(\mathbb{C}^*)^k -  \bigcup_{i=1}^m\{P_i=0\}$ in which the integral (\ref{eq-generalized_euler_integral}) depends only on the homology class of $\Sigma$. This means that we may think of the generalized Euler integral as a generically multivalued, analytic function of the coefficients $v_\omega$. 
\par Write $\alpha = (\alpha_1,\dots,\alpha_m) \in \mathbb{C}^m$ and $\beta = (\beta_1,\dots,\beta_k) \in \mathbb{C}^k$. Then we may write $\Phi(v) = \Phi(\alpha;\beta \, \vert \, v_{\omega})$ for the multivalued analytic function defined by (\ref{eq-generalized_euler_integral}). We are interested in finding a minimal set of linear differential operators in the variables $v_\omega$ that annihilate $\Phi$. To each $P_i$, represent the set $\mathcal{A}_i$ of powers of its monomials by the $(k \times j_i)$ integer matrix $\mathcal{A}_i = \left(\omega_{pq}\right)$ where $p = 1,\dots,k$ and $q=1,\dots,j_i = |\mathcal{A}_i|$. In accordance with the ``Cayley trick'' \cite[\S 2.5]{MR1080980}, we augment the columns of the matrix $\mathcal{A}_i$ with copies of the standard basis vector $e_i$ of the lattice $\mathbb{Z}^m$, defining a new matrix $a_i = e_i \times \mathcal{A}_i$ of size $(m+k \times j_i)$. It is convenient to collect the data for all the polynomials into the matrix $\mathcal{A} = \left(a_{1} \; \cdots \; a_{m}\right)$ of size $(m + k \times j_1 + \cdots + j_m)$. Set $n= m+k$ and $N = j_1 + \cdots + j_m$. We assume as in \cite[\S 1.1]{MR1080980} that $\mathcal{A} \subset \mathbb{Z}^n$ generates the lattice as an abelian group. Further, we assume that $\mathcal{A}$ is contained in an integral hyperplane in $\mathbb{Z}^n$; that is to say that there is a $\mathbb{Z}$-linear map $H : \mathbb{Z}^n \to \mathbb{Z}$ such that $H(\mathcal{A})=1$. From the context of the Cayley trick, the map $H$ is simply the induced map on $\mathbb{Z}^n$ by adding the entries of the first $m$ rows of the columns of $\mathcal{A}$.
\par The lattice of relations for our system is given as $\mathbb L = \operatorname{null}(\mathcal{A}) \cap \mathbb{Z}^N$ by definition. We think of elements $\ell \in \mathbb{L} \subset (\mathbb{Z}^N)^*$ as row vectors, and take $\mathbb{L} = \{\ell^1,\dots,\ell^j\}$ as a generating set, where $j = N - n$. For any $\ell = (\ell_1,\dots,\ell_n) \in \mathbb{L}$, the assumption on the hyperplane $H$ ensures that $\sum_{k=1}^{N} \ell_k = 0$. These facts allow us to derive the desired differential equations. Consider $v=(v_1,\dots,v_{N})\in \mathbb{C}^N$ and put $\theta_i = v_i \frac{\partial}{\partial v_i}$. For $\ell^i = (\ell^i_1  \dots, \ell^i_N)$, define 
\begin{equation}
	\label{eq-box_operator}
	\scalemath{1}{\square_i = 
	\prod_{\ell^i \,: \,\ell^i_j >0}\left(\frac{\partial}{\partial v_i}\right)^{\ell^i_j}-\prod_{\ell^i \, : \, \ell^i_j <0}\left(\frac{\partial}{\partial v_i}\right)^{-\ell^i_j} \; .}
\end{equation}
By the considerations above, each $\square_i$ is a homogeneous differential operator. Put $\Theta = (\theta_1,\dots,\theta_N)$, and define $\vec{\gamma} = (\alpha_1,\dots,\alpha_m, -\beta_1-1,\dots,-\beta_k-1) \in \mathbb{C}^n$ from the data of the generalized Euler integral (\ref{eq-generalized_euler_integral}).
\par The so-called $\mathcal{A}$-hypergeometric system \cite[Def. 1.2, Thm. 2.7]{MR1080980} with parameters $\gamma$ is then the system of differential equations that has the following effect on the function $\Phi(v)$:
\begin{equation}
	\label{eq-A_hypergeometric_system}
	\scalemath{1}{\square_i\Phi(v) = 0\, , \quad i = 1,\dots,j \, ,\quad \quad \mathcal{A}\Theta\Phi(v) = \gamma \Phi(v).}
\end{equation}

The meaning of the equations (\ref{eq-A_hypergeometric_system}) in the context of the generalized Euler integral $\Phi(v)$ is as follows. Such systems were defined generally in \cite{MR0902936}, where they were shown to be holonomic, so the number of linearly independent solutions at a generic point is finite. Direct calculation shows that (\ref{eq-generalized_euler_integral}) is annihilated by the $\square_i$ operators. Regarding $Z_k$, $k=1,\dots,n$, as the $k$th row of the vector $\mathcal{A}\Theta$, the  equation $Z_k\Phi(v) = \gamma_k \Phi(v)$ in (\ref{eq-A_hypergeometric_system}) is a quasi-homogeneity condition coming from the freedom to rescale the integrand (\ref{eq-generalized_euler_integral}) by the torus $(\mathbb{C}^*)^n = (\mathbb{C}^*)^k \times (\mathbb{C}^*)^m$ in the natural way on the variables $x_1,\dots,x_k$ and the polynomials $P_1,\dots,P_m$, respectively. Ultimately quotienting by the torus action allows us to use the nonhomogeneous relations from the $Z_k$ operators together with the homogeneous $\square_i$ relations to reduce the full system (\ref{eq-A_hypergeometric_system}) to a suitable homogeneous subsystem in a smaller number of variables. 
\subsubsection{The $F_2$ hypergeometric system}
As an example, we calculate the hypergeometric differential equations for $F_2$ using the GKZ perspective. These equations are well-known (see for example \cite[\S 6.2]{yoshida1987}), but their derivation from the GKZ perspective applied to their integral representations above does not appear in the literature to the best of our knowledge. 
Recall the integral representation of $F_2$ from (\ref{eq-appell_integral_rep}).
Our polynomials and their corresponding $\mathcal{A}_i$ matrices are (note that we are fixing an order on the variables ~$(u,v)$)
\begin{equation*}
	\scalemath{.9}{1-u + 0u^0v^0  \leftrightarrow
	\begin{pmatrix}
		0 & 1 & 0\\
		0 & 0 & 0
	\end{pmatrix} , \quad
	1-v + 0u^0v^0 \leftrightarrow
	\begin{pmatrix}
		0 & 0 & 0\\
		0 & 1 & 0
	\end{pmatrix} ,\quad
	1-x u - y v  \leftrightarrow
	\begin{pmatrix}
		0 & 1 & 0\\
		0 & 0 & 1
	\end{pmatrix} .}
\end{equation*}
After applying the Cayley trick and combining, we get
\begin{equation*}
\scalemath{.7}{	\mathcal{A} =
	\begin{pmatrix}
		1 & 1 & 1 & 0 & 0 & 0 & 0 & 0 & 0 \\
		0 & 0 & 0 & 1 & 1 & 1 & 0 & 0 & 0 \\
		0 & 0 & 0 & 0 & 0 & 0 & 1 & 1 & 1 \\
		0 & 1 & 0 & 0 & 0 & 0 & 0 & 1 & 0 \\
		0 & 0 & 0 & 0 & 0 & 1 & 0 & 0 & 1
	\end{pmatrix} \, ,  \quad \mathbb L =
	\begin{pmatrix}
		0 & 0 & 0 & 1 & 0 & -1 & -1 & 0 & 1 \\
		1 & -1 & 0 & 0 & 0 & 0 & -1 & 1 & 0 \\
		0 & 0 & 0 & -1 & 1 & 0 & 0 & 0 & 0 \\
		-1 & 0 & 1 & 0 & 0 & 0 & 0 & 0 & 0
	\end{pmatrix}.}
\end{equation*}
Moreover we get the vector $	\vec{\gamma} = (\gamma_1 - \beta_1 -1, \gamma_2 - \beta_2 -1,-\alpha,-\beta_1,-\beta_2)  \in \mathbb{C}^5$. Thus the full GKZ system in (\ref{eq-A_hypergeometric_system}) is (suppressing $\Phi$)
\begin{equation}
	\label{eq-F2_GKZ_system}
	\scalemath{.75}{\begin{cases}
		\theta_1 + \theta_2 + \theta_3 = \gamma_1 - \beta_1 -1 , \\
		\theta_4 + \theta_5 + \theta_6 = \gamma_2 - \beta_2 -1 , \\
		\theta_7 + \theta_8 + \theta_9 = -\alpha , \\
		\theta_2 + \theta_8 = -\beta_1, \\
		\theta_6 + \theta_9 = -\beta_2,
	\end{cases}}
	\quad
	\scalemath{.9}{\begin{cases}
		\frac{\theta_4}{v_4} = \frac{\theta_5}{v_5} , \\
		\frac{\theta_1}{v_1} = \frac{\theta_3}{v_3},\\
		\frac{\theta_4\theta_9}{v_4v_9} = \frac{\theta_6\theta_7}{v_6v_7},\\
		 \frac{\theta_1\theta_8}{v_1v_8} = \frac{\theta_2\theta_7}{v_2v_7} .
	\end{cases}}
\end{equation}
Products in the second order equations should be understood as composition of the corresponding operators, and we have rewritten the $\square_i$ equations in terms of the logarithmic derivative operators $\theta_i$, using the fact that the $\theta$ operators\footnote{Notice that in general one will encounter terms like $(\pd{v_i})^2$ in the $\square$ equations, whence one must use $(\pd{v_i})^2 = \frac{1}{v_i^2}(\theta_i^2-\theta_i)$, etc.} commute $\theta_i\theta_j = \theta_j\theta_i$. We then introduce a change of variables to eventually reduce the system by dividing by the torus action. This will give us two variables that will end up being $(x,y)$ in the definition of $F_2$. 

From the second order equations above, set $x = \frac{v_1v_8}{v_2v_7}$ and $y = \frac{v_4v_9}{v_6v_7}$. This choice of variables comes precisely from the ``second order'' lattice vectors $\ell^1,\ell^2 \in \mathbb{L}$, where the entries $\ell^i_j$ determine $x$ and $y$ as $x = \prod_{j=1}^{9} v_j^{\ell^1_j}$, $y = \prod_{j=1}^{9} v_j^{\ell^2_j}$. Let $\theta_x = x \pd{x}$ and $\theta_y = y \pd{y}$. Then an easy computation shows that $\theta_8 = \theta_x$ and $\theta_9 = \theta_y$. Hence the second order equations above can be written as $\theta_1\theta_x - x\theta_2\theta_7 = 0,$ and $\theta_4\theta_y - y\theta_6\theta_7 = 0$.

Using the first order $\square_i$ equations as well as the remaining equations from the $Z_k$ operators, all other derivatives can be expressed in terms of $\theta_x$ and $\theta_y$ alone, after we set $v_1=1,v_3=0,v_4=1,v_5=0$. This is permissible because of the relevant torus action. We find the following:
\begin{equation*}
	\scalemath{.85}{\begin{cases}
		\theta_1 = \theta_x+\gamma_1-1 \,,\\
		\theta_2 = -\theta_x-\beta_1 \,,\\
		\theta_4 = \theta_y+\gamma_2-1 \,, 
	\end{cases}}
	\quad \quad
	\scalemath{.85}{\begin{cases}
	\theta_6 = -\theta_y-\beta_2  \,,\\
	\theta_7 = -\theta_x-\theta_y-\alpha \, .
	\end{cases}}
\end{equation*}
Substituting these expressions into the second order equations with $\theta_x$ and $\theta_y$, we find
\begin{equation*}
	\scalemath{1}{\begin{cases}
		(1-x)\theta_x^2-x\theta_x\theta_y+(\gamma_1-1-(\alpha+\beta_1)x)\theta_x -x\beta_1\theta_y - x\alpha\beta_1 =0 \, , \\ 
		(1-y)\theta_y^2-y\theta_x\theta_y+(\gamma_2-1-(\alpha+\beta_2)y)\theta_y -y\beta_2\theta_x - y\alpha\beta_2 =0 \, . 
	\end{cases}}
\end{equation*}
Expanding these out, we arrive at
\begin{equation}
	\label{eq-F2_hypergeometric_eqns}
\scalemath{1}{	\begin{cases}
		x\left(1-x\right) \frac{\partial^2 F}{\partial x^2}-x y \frac{\partial^2 F}{\partial x \partial y}+\left(\gamma_1-\left(\alpha+\beta_1+1\right) x\right) \frac{\partial F}{\partial x}-\beta_1 y \frac{\partial F}{\partial y}-\alpha \beta_1 F=0 \, , \\
		y\left(1-y\right) \frac{\partial^2 F}{\partial y^2}-x y \frac{\partial^2 F}{\partial x \partial y}+\left(\gamma_2-\left(\alpha+\beta_2+1\right) y\right) \frac{\partial F}{\partial y}-\beta_2 x \frac{\partial F}{\partial x}-\alpha \beta_2 F=0 \, .
	\end{cases}}
\end{equation}
Then (\ref{eq-F2_hypergeometric_eqns}) is precisely the well-known Fuchsian system annihilating $F_2$. Hence we identify
\begin{equation*}
	\scalemath{1}{F_2\left(\left.
	\genfrac{}{}{0pt}{}{\alpha ; \beta_1, \beta_2}{\gamma_1, \gamma_2}
	\right|\, x, y\right) \overset{\cdot}{=} \Phi\left(\vec{\gamma} \, \Big\vert \, 1,1,0,1,0,1,1,x=\frac{v_1v_8}{v_2v_7},y = \frac{v_4v_9}{v_6v_7}\right) \, ,}
\end{equation*}
where $\overset{\cdot}{=}$ signifies equality up to a nonvanishing constant.

In a completely analogous way, one recovers the well-known Fuchsian system that annihilates $F_4$:
\begin{equation}
	\label{eq-F4_hypergeometric_eqns}
	\scalemath{.8}{\begin{cases}
		x(x+y-1)\frac{\partial^2F}{\partial x^2} +2 xy\frac{\partial^2F}{\partial y \partial x} +\left((\alpha +\beta +1) x+\gamma_1( y-1)\right)\frac{\partial F}{\partial x}+y\left(\alpha+\beta-\gamma_2+1\right)\frac{\partial F}{\partial y}+\alpha \beta F = 0 \, ,\\
		y(x+y-1)\frac{\partial^2 F}{\partial y^2} + 2x y\frac{\partial^2 F}{\partial y \partial x}+x\left(\alpha+\beta-\gamma_1+1\right)\frac{\partial F}{\partial x} +\left((\alpha +\beta +1)y+\gamma_2 (x-1)\right)\frac{\partial F}{\partial y} +\alpha \beta F = 0 \, .
	\end{cases}}
\end{equation}
Both (\ref{eq-F2_hypergeometric_eqns}) and (\ref{eq-F4_hypergeometric_eqns}) are linear systems of the form (\ref{eq-rank_4_system}) studied in \S \ref{ss-rank-4_systems}. The following statement is well-known, but crucial. 

\begin{proposition}
	\label{prop-F2_F4_integrable_systems}
	The systems (\ref{eq-F2_hypergeometric_eqns}) and (\ref{eq-F4_hypergeometric_eqns}) annihilating Appell's $F_2$ and $F_4$ functions each satisfy the Maurer-Cartan equation $d\omega - \omega \wedge \omega = 0$. Hence the differential systems are of rank-4, and are thus integrable systems of finite type.
\end{proposition}

\section{Hypergeometric line congruences}
\label{s-hypergeometric_congruences}
\par It follows from Proposition \ref{prop-F2_F4_integrable_systems} that both (\ref{eq-F2_hypergeometric_eqns}) and (\ref{eq-F4_hypergeometric_eqns}) define surfaces $S_{F_2}$ and $S_{F_4}$ in projective space. These surfaces are rather interesting since they intrinsically depend on the auxiliary parameters of the hypergeometric functions given in \S \ref{ss-integral_reps}. More generally, the surfaces $S_{F_2}$ and $S_{F_4}$ have been studied in some detail by Sasaki \cite{MR1858701}, who was able to compute systems of asymptotic coordinates for both surfaces. There, Sasaki showed that $S_{F_4}$ was a projectively applicable surface depending on precisely three parameters that are a specific combination of $\alpha,\beta,\gamma_1,\gamma_2$. This puts the surface in the category of Proposition \ref{prop-vanishing_invariants} (\textit{iv}). A similar analysis shows that $S_{F_2}$ is also projectively applicable with precisely two free parameters determined by certain combinations of the $\alpha,\beta_1,\beta_2,\gamma_1,\gamma_2$.

However, an analysis of these surfaces in conjugate coordinates, and hence of the associated Laplace transforms / invariants, or line congruences has not been done yet to the best of our knowledge. The closest to such an analysis is in \cite[Appx. A]{MR2216951}, where Sasaki attempted to compute the negative Laplace transform of the $F_2$ system. Our Theorem \ref{thm-F2_W_congruence} corrects an error present in that work.

Though these surfaces have geometry that is quite distinct, it turns out that their conformal structures are equivalent \cite[\S 5.4]{MR960834}; this is what allowed Sasaki's analysis of the corresponding surfaces in \cite{MR1858701}. Specifically, it follows from (\ref{eq-F2_hypergeometric_eqns}) and (\ref{eq-F4_hypergeometric_eqns}) that the conformal structures $[h_2]$ and $[h_4]$ of the surfaces $S_{F_2}$ and $S_{F_4}$ are generated respectively by the tensors
\begin{equation}
	\label{eq-F2_conformal}
	\scalemath{1}{h_2= y d x \otimes d x+ (1-x)(1-y)(d x \otimes d y+d y \otimes d x) + x d y \otimes d y \, ,}
\end{equation}
\begin{equation}
	\label{eq-F4_conformal}
	\scalemath{1}{h_4 = 2y d x \otimes d x+\left(1-x-y\right) (d x \otimes d y + d y \otimes d x)+2x d y \otimes d y \, .}
\end{equation}
If we change variables $(u=x,v=y)$ on $S_{F_4}$, then Sasaki \& Yoshida showed in loc. cit. that the rational map $T : S_{F_2} \to S_{F_4}$ defined by
\begin{equation}
	\label{eq-F2_F4_equivalence}
	\scalemath{1}{T \; : \quad (u,v)=\left(\frac{x^2}{(x+y-2)^2}, \frac{y^2}{(x+y-2)^2}\right)}
\end{equation}
satisfies $T^*[h_4]=[h_2]$. 
\par Key for us is the following formula representing $F_2$ as a certain two-parameter Euler integral transform of $_2F_1$ due to Clingher, Doran, \& Malmendier \cite[Corollary 2.2]{MR3767270}.
\begin{lemma}
	\label{lem-F2_integral_transform}
	For $\mathrm{Re}(\gamma_1) > \mathrm{Re}(\beta_1) > 0$ and $\mathrm{Re}(\gamma_2) > \mathrm{Re}(\beta_2) > 0$, the following integral relation holds between the Gauss hypergeometric function $_2F_1$ and Appell's function $F_2$:
	\begin{equation}
		\label{eq-F2_integral_transform}
	\scalemath{.8}{ F_2\left(\left.\begin{array}{c}
			\alpha ; \beta_1, \beta_2 \\
			\gamma_1, \gamma_2
		\end{array} \right\rvert\, \frac{1}{s}, 1-\frac{t}{s}\right)=-\frac{\Gamma\left(\gamma_2\right)s^\alpha(s-t)^{1-\gamma_2}}{\Gamma\left(\beta_2\right) \Gamma\left(\gamma_2-\beta_2\right)}
		\int_s^t \frac{d u}{(s-u)^{1-\beta_2}(u-t)^{1+\beta_2-\gamma_2} u^\alpha}{ } \, _2 F_1\left(\left.\begin{array}{c}
			\alpha, \beta_1 \\
			\gamma_1
		\end{array} \right\rvert\, \frac{1}{u}\right) .}
	\end{equation}
\end{lemma}
Equation (\ref{eq-F2_integral_transform}) has a number of surprising consequences that are explored in \cite{MR3767270}, as well as by the second author \cite[\S 5.3.2, Prop. 6.2.92]{schultz_geometry_2021}. Most significantly for us, it produces a system of conjugate coordinates on the surface $S_{F_2}$:
\begin{lemma}
	\label{lem-F2_conjugate}
	The rational map $T_2 : S_{F_2} \to S_{F_2}$ given by 
	\begin{equation}
		\label{eq-F2_conjugate_coords}
		\scalemath{1}{T_2 \; : \quad (x,y) = \left(\frac{1}{s},1-\frac{t}{s}\right)}
	\end{equation}
	produces a system of conjugate coordinates on $S_{F_2}$.
\end{lemma}
\begin{proof}
	By direct computation, the pullback $T_2^*h_2$ is conformal to the diagonal tensor $	d s \otimes d s-\frac{s(1-s)}{t(1-t)} d t \otimes d t$. It follows from Definition \ref{def-conjugate_coords3} that the new coordinates are conjugate on $S_{F_2}$.
\end{proof}

Since the surfaces $S_{F_2}$ and $S_{F_4}$ are conformally equivalent via (\ref{eq-F2_F4_equivalence}), it follows that conjugate coordinates may be produced on $S_{F_4}$ by composition.
\begin{lemma}
	\label{lem-F4_conjugate}
	The rational map $T_4 : S_{F_4} \to S_{F_4}$ given by 
	\begin{equation}
		\label{eq-F4_conjugate_coords}
	\scalemath{1}{	T_4 \; : \quad (x,y) = \left(\frac{1}{(s+t)^2},\frac{(s-t)^2}{(s+t)^2}\right)}
	\end{equation}
	produces a system of conjugate coordinates on $S_{F_4}$.
\end{lemma}
\begin{proof}
	By direct computation, $T_4^*h_4$ is conformal to $d s \otimes d s-\frac{4 s^2-1}{4 t^2-1} dt \otimes dt$.
\end{proof}
\begin{remark}
	The mapping $T_4$ is the composition $	T_4 = \widetilde{T} \circ T \circ T_2$, where $T$ is the map in (\ref{eq-F2_F4_equivalence})  and $\widetilde{T} (s,t) = (s+\frac{1}{2},t+\frac{1}{2})$ and we identify the domain as a subset of $S_{F_4}$ using the fact that $T$ is a local conformal isomorphism.
\end{remark}
\subsection{Appell's $F_2$ in conjugate coordinates}
\label{s-f_2_conjugate_and_laplace}
Under the change of coordinates in Lemma \ref{lem-F2_conjugate} together with the projective gauge transformation $z \mapsto s^\alpha(s-t)^{1-\gamma_2} z$, the conjugate system for $F_2$ has coefficients
\begin{equation*}
	\scalemath{.7}{\begin{cases}
		a = -\frac{1-\gamma_2+\beta_2}{s-t} \\
		b = -\frac{\beta_2-1}{s-t} \\
		c = 0 
	\end{cases}} \quad
	\scalemath{.7}{\begin{cases}
		q = \frac{(s-1) s}{(t-1) t} \\
		m = \frac{\left(2 \alpha-\gamma_1-2 \gamma_2+4\right) s^2+\left(-2 \alpha+2 \beta_2+\gamma_1-2\right) s t+\left(-\alpha+\beta_1-\beta_2+2 \gamma_2-3\right) s+\left(\alpha-\beta_1-\beta_2+1\right) t}{t(t-1)(s-t)}\\
		n = \frac{\left(2 \alpha+2 \beta_2-\gamma_1-2 \gamma_2+2\right) s t+\left(-\alpha+\beta_1-\beta_2+\gamma_2-1\right) s+\left(-2 \alpha+\gamma_1+2 \gamma_2-4\right) t^2+\left(\alpha-\beta_1-\beta_2-\gamma_2+3\right) t}{t(t-1)(s-t)}\\
		r = \frac{\left(-\alpha+\gamma_2-1\right)\left(\gamma_2+\gamma_1-\alpha-2\right)}{(t-1) t} \\
	\end{cases}}
\end{equation*}
Notice that the gauge factor is precisely the nonconstant part of the multiplicative factor appearing (\ref{eq-F2_integral_transform}). We see as well that the hyperbolic equation (\ref{eq-conjugate_eq1}) for $F_2$ is precisely the Euler-Poisson-Darboux equation (\ref{eq-EPD}). 

\begin{theorem}
	\label{thm-F2_laplace}
	The sequence of Laplace invariants for $F_2$ is equivalent to those of the Euler-Poisson-Darboux equation (\ref{eq-EPD}).
\end{theorem}
\begin{proof}
Under the change of parameters $\beta \to \gamma_2-\beta_2$, $\beta' \to \beta_2$, the EPD equation (\ref{eq-EPD}) becomes precisely the form of the hyperbolic equation for $F_2$ above, yielding
\begin{equation*}
	\scalemath{1}{\mathsf{h} = \frac{\left(\beta_2-\gamma_2+1\right) \beta_2}{(s-t)^2} \, ,
	\quad
	\mathsf{k} = \frac{\left(\beta_2-1\right)(\beta_2 - \gamma_2)}{(s-t)^2}} 
\end{equation*}
as the first-order invariants of $F_2$, equivalent to (\ref{eq-EPD_laplace_invariants}).
\end{proof}

\begin{corollary}
	\label{cor-F2_laplace_invariants}
	The sequence of Laplace invariants for $F_2$ is determined by 
	\begin{equation}
\scalemath{1}{\mathsf{h}_n = \frac{(n- 1 + \beta_2)(n+\beta_2 - \gamma_2)}{(s-t)^2}
		\quad\quad
		\mathsf{k}_n = \frac{(n- 2 + \beta_2)(n-1+\beta_2 - \gamma_2)}{(s-t)^2}}
	\end{equation}
\end{corollary}
\begin{proof}
	By induction on $n$ using the recursive relationships in (\ref{eq-higher_laplace}).
\end{proof}

\begin{remark}
	\label{rmk-F2_vanishing_laplace}
	We can easily determine when the (higher) Laplace invariants of $F_2$ vanish in terms of the parameters $\beta_2,\gamma_2$.
\end{remark}

The meaning of these results can be understood completely in terms of the integral representation (\ref{eq-F2_integral_transform}), since the  EPD equation (\ref{eq-EPD}) admits a Riemann function in terms of the Gauss hypergeometric function $_2F_1$ \cite[\S 4.6]{MR2216951} as well as solutions of the form $z(x,y) = (u-x)^\beta (u-y)^{-\beta^\prime}$
which appear in the ``splitting" of the integrand of (\ref{eq-F2_integral_transform}) for the values of $\beta,\beta^\prime$ given in Theorem \ref{thm-F2_laplace}. Moreover, we have the following by direct computation using (\ref{eq-positive_weingarten}, \ref{eq-negative_weingarten}):

\begin{theorem}
	\label{thm-F2_W_congruence}
	For all values of the parameters $\alpha,\beta_1,\beta_2,\gamma_1,\gamma_2$, both the positive and negative congruences $\mathbf{L}^\pm = \{S_{F_2},S_{F_2}^\pm\}$ are $W$-congruences. 
\end{theorem}

This corrects an error in \cite[Appx. A]{MR2216951}, where it was claimed that the negative congruence $\mathbf{L}^-$ is a $W$-congruence if and only if $\beta_2=0$. There, the author attempts to compute the negative Laplace invariants of $F_2$ by utilizing the commuting vector fields $X = (x-1)\partial_x$, $Y =  X + y\partial_y$ on $S_{F_2}$ which satisfy $h_2(X,Y)=0$ and rewriting the $F_2$ system (\ref{eq-F2_hypergeometric_eqns}) in terms of $X,Y$. However, one can show that up to the allowable changes of coordinates described in \S \ref{sss-laplace_invariants}, by straightening the commuting vector fields $X,Y$, one recovers precisely the conjugate coordinates from Lemma \ref{lem-F2_conjugate}. 

\subsection{Appell's $F_4$ in conjugate coordinates}
\label{ss-f_4_conjugate_and_laplace}
Using Lemma \ref{lem-F4_conjugate} together with the projective gauge transformation $z \mapsto (s+t)^{\alpha+\beta-\frac{1}{2}}(s-t)^{-\gamma_2+\frac{1}{2}}z$, the conjugate system for $F_4$ has coefficients
\begin{equation*}
\scalemath{.7}{\begin{cases}
		a = 0 \\
		b = 0 \\
		c = \frac{(\gamma_2-\frac{1}{2})(\gamma_2-\frac{3}{2}) }{(s-t)^2} - \frac{(\alpha-\beta +\frac{1}{2})(\alpha-\beta -\frac{1}{2})}{(s+t)^2} 
	\end{cases}} 
\scalemath{.60}{\begin{cases}
		q = \frac{4s^2-1}{4t^2-1} \\
		m  \frac{4s\left(2 (\alpha+ \beta- \gamma_1- \gamma_2)+3\right)}{(2 t-1)(2 t+1)}\\
		n = \frac{4t\left(2 (\alpha+ \beta- \gamma_1- \gamma_2)+3\right)}{(2 t-1)(2 t+1)} \\
		r =\frac{\left(\alpha+\beta-\gamma_2\right)\left(\alpha+\beta-2 \gamma_1-\gamma_2+2\right)}{(2 t-1)(2 t+1)} - \frac{2t(2\gamma_2-1)(2\gamma_2-3)}{(s-t)(2 t-1)(2 t+1)} - \frac{1}{2} \frac{(2 \alpha-2 \beta+1)(2 \alpha-2 \beta-1)}{(s+t)^2} - \frac{1}{2} \frac{(2\gamma_2-1)(2\gamma_2-2)}{(s-t)^2} + \frac{2t(2 \alpha-2 \beta+1)(2 \alpha-2 \beta-1)}{(s+t)(2 t-1)(2 t+1)}\\
	\end{cases}}
\end{equation*}

\begin{theorem}
	\label{thm-F4_laplace_invariants}
	The sequence of Laplace invariants for $F_4$ is equivalent to those for Darboux's Harmonic equation (\ref{eq-harmonic_equation}).
\end{theorem}
\begin{proof}
	This is exactly analogous to Theorem \ref{thm-F2_laplace}. The relation between the parameters in (\ref{eq-harmonic_equation}) is $\alpha \to \gamma_2 - \frac{1}{2}$ and $\beta \to \alpha - \beta + \frac{1}{2}$.
\end{proof}

This result for $F_4$ is in perfect agreement with Iwasaki \cite{MR946650}, who showed that Darboux's Harmonic equation admits $F_4$ as a Riemann function. It would be interesting to determine if there is also an interpretation in terms of an integral transform. 

\begin{theorem}
	\label{thm-F4_W_congruence}
	For all values of the parameters $\alpha,\beta,\gamma_1,\gamma_2$, both the positive and negative congruences $\mathbf{L}^\pm = \{S_{F_4},S_{F_4}^\pm\}$ are $W$-congruences. 
\end{theorem}

As mentioned in Remark \ref{rmk-geometry_of_laplace}, it is of interest to determine when for example a line congruence can be generated by two focal quadrics. In the case of $F_2$ and $F_4$, we have the following negative result:

\begin{corollary}
	\label{cor-focal_quadrics}
	Let $S = S_{F_2}$ or $S_{F_4}$. The only values of the respective parameters for which the positive congruence $\mathbf{L}^\pm = \{S,S^\pm\}$ is generated by focal quadrics correspond to parameter values for which $h=0$ in the case of $\mathbf{L}^+$ and $\mathsf{k}=0$ in the case of $\mathbf{L}^-$.
\end{corollary}
\begin{proof}
	Evaluating the vanishing of the cubic form $\Phi$ from Theorem \ref{thm-conjugate_invariants} applied to the conjugate form of $F_2$ and $F_4$ above reveals that the surfaces $S_{F_2}$ and $S_{F_4}$ are quadrics precisely when for $F_2 : \{\alpha=\beta_1+\beta_2-\tfrac{1}{2}, \gamma_1=2 \beta_1, \gamma_2=2 \beta_2\}$ and for $F_4 : \left\{\gamma_2=\alpha+\beta-\gamma_1+1\right\}.$ This matches Sasaki \& Yoshida, \cite[\S 5.5]{MR960834}. Utilizing the relevant formulae (\ref{eq-pos_quad_quad}, \ref{eq-neg_quad_quad}) in Appendix \ref{appendix} that determine when $S,S^\pm$ are simultaneously quadrics yields the rest of the statement.
\end{proof}
\bibliographystyle{abbrv}
\bibliography{ryan_schultz}
\newpage
\appendix
\section{Laplace transform components}
\label{appendix}
Here we present the components of the positive and negative Laplace transformed systems as in Theorems \ref{thm-positive_laplace_system} and \ref{thm-negative_laplace_system}. Then we present the differential constraints utilized in Corollary \ref{cor-focal_quadrics} for when both focal surfaces of the line congruences $\mathbf{L}^\pm = \{S,S^\pm\}$ are quadric surfaces. 

\begin{flushleft}
	\textbf{Theorem \ref{thm-positive_laplace_system}}: The components $m_1,n_1,r_1$ of the positive Laplace transformed system (\ref{eq-positive_rank-4_system}) are given as
	\begin{equation}
		\label{eq-pos_components1}
	\scalemath{.75}{\begin{cases}
			m_1 =-\frac{\left(\left(2a_x b+a_{x x}-\mathsf{h}_x\right)+\mathsf{h}_x\left(-2 a_x+\mathsf{h}\right)\right) q}{\mathsf{h}^2} +\frac{(\left(m b+2 a_y-n_y\right) a+bm_y-m c-a_yn-a_{y y}+r_y)\mathsf{h}-m_y\mathsf{h}_x}{\mathsf{h}^2}  -\frac{\left(\mathsf{h}\left(a^2-a n+m b+r-a_y\right)-m\mathsf{h}_x\right)q_y}{\mathsf{h}^2 q} \; , \\
			n_1 = n - (\log q)_y \; , \\
			r_1  = \frac{r_{11}q}{\mathsf{h}^2} + \frac{r_{12}}{\mathsf{h}^2} + \frac{r_{13}q_y}{q\mathsf{h}^2} \; ,
		\end{cases}}
	\end{equation}
	where $r_{11},r_{12},r_{13}$ are defined as 
	\begin{equation}
		\label{eq-pos_components2}
	\scalemath{.75}{\begin{cases}
			r_{11} = \left(-b^2+b_x\right) \mathsf{h}^2+\left(a b^3-2 ab_x b-c b^2+a_xb_x+cb_x+bc_x\right) \mathsf{h} +b\mathsf{h}_x\left(b a-a_x-c\right) \, ,\\
			r_{12} = \left(-m b-2 a_y+n_y+r\right)\mathsf{h}^2 +\left(b^2 m a+\left(\left(2 a_y-n_y\right) a-m c-a_y n+r_y-a_{y y}\right) b+m_yb_x\right) \mathsf{h} -bm_y\mathsf{h}_x \; , \\
			r_{13} = (n-a) \mathsf{h}^2+\left(\left(a^2-a n-a_y+r\right) b+b_x m\right) \mathsf{h}-b m\mathsf{h}_x \, .
		\end{cases}}
	\end{equation}
\end{flushleft}
\begin{flushleft}
	\textbf{Theorem \ref{thm-negative_laplace_system}}: The components $m_0,n_0,r_0$ of the negative Laplace transformed system (\ref{eq-negative_rank-4_system}) are given as
	\begin{equation}
		\label{eq-neg_components1}
		\scalemath{.8}{m_0 = \frac{\mathsf{k}\left(m+q_x\right)}{n_x-2 b_y+\mathsf{k}} \, , \quad
			n_0 = \frac{n_{01}}{\mathsf{k}(n_x-2 b_y+\mathsf{k})} \; , \quad
			r_0 = \frac{r_{01}}{\mathsf{k}(n_x-2 b_y+\mathsf{k})} \; ,}
	\end{equation}
	where $n_{01}$ and $r_{01}$ are defined by 
	\begin{equation}
		\label{eq-neg_components2}
	\scalemath{.8}{\begin{cases}
			n_{01} = \left(b^2 q_x+\left(2 b_x q-a^2+a n-m_x+a_y\right) b-b_x q_x-b_x m+a c-n c-b_{xx} q-c_y+r_x\right)\mathsf{k} \\
			\;\;\;\;\;\;+(n_x-2 b_y+\mathsf{k})\left(b a^2+\left(b_y-c\right) a-\mathsf{k}_y\right) \; , \\
			r_{01} = \left(-2 b_x q+a^2-a n-b q_x+r+m_x-a_y\right) \mathsf{k}^2 \\
			\;\;\;\;\;\;+\left(-2 a^3 b+\left(b n+b_y+2 c-n_x\right) a^2 +\left(b^2 q_x+\left(2 b_x q-m_x+2 a_y\right) b-b_x q_x-b_x m \right. \right.\\
			\;\;\;\;\;\;\left.\left.-n c-b_{x x} q-c_y+r_x \right)a -a_y\left(b_y+c-n_x\right)\right)\mathsf{k} \\
			\;\;\;\;\;\;+(n_x-2 b_y+\mathsf{k})\left(b a^2+\left(b_y-c\right) a-\mathsf{k}_y\right) \; .
		\end{cases}}
	\end{equation}	
\end{flushleft}
\begin{flushleft}
	\textbf{Remark \ref{rmk-geometry_of_laplace}}: Let $\Phi^+$ be the cubic form of the positive Laplace transformed surface $z^+$ with components $A^+$ and $B^+$ defined in the obvious way from Theorem \ref{thm-conjugate_invariants} and the transformed system (\ref{eq-positive_rank-4_system}). Then we have that both surfaces $S,S^+$ are both quadrics if and only if $A = 0 = B$ and $A^+ = 0 = B^+$, which happens if and only if 
	\begin{equation}
		\label{eq-pos_quad_quad}
	\tiny	\begin{aligned}
			&4 q_y^2 a \mathsf{h}+q_y q_x^2 \mathsf{h}-q_y q_x \mathsf{h}_x q-2 q_y q_x b q \mathsf{h}-q_y q \mathsf{h} q_{x x}+\left(-4 a^2-6 a_y+4 r\right) q_y q \mathsf{h}-q_x q_{x y} q \mathsf{h} \\
			& +\left(4 b_y-2 a_x\right) q_x q^2 \mathsf{h}+\mathsf{h}_x q_{x y} q^2+\left(2 b a+4 b_y-2 a_x-2 c\right) \mathsf{h}_x q^3+2 q_{x y} b q^2 \mathsf{h}+\left(8 b_y b-2 b_x a\right. \\
			& \left.-2 b a_x-2 a_{x x}+2 c_x+4 b_{x y}\right) q^3 \mathsf{h}+\left(8 a_y a+4 a_{y y}-4 r_y+q_{x x y}\right) q^2 \mathsf{h}-2 q_{y y} a q \mathsf{h} = 0 \, ,
		\end{aligned}
	\end{equation}
	\begin{equation*}
	\tiny	\begin{aligned}
			& 3 \mathsf{h}_y q_y q_x q-3 \mathsf{h}_y q_{x y} q^2+\left(-6 b a-12 b_y+6 a_x+6 c\right) \mathsf{h}_y q^3-4 q_y^2 q_x \mathsf{h}+4 q_y q_{x y} q \mathsf{h}+\left(4 b a+8 b_y\right. \\
			& \left.-4 a_x-4 c\right) q_y q^2 \mathsf{h}+q_{y y} q_x q \mathsf{h}+\left(-2 b_y a-2 b a_y-4 b_{y y}+2 a_{x y}+2 c_y\right) q^3 \mathsf{h}-q_{x y y} q^2 \mathsf{h} = 0 \, .
		\end{aligned}
	\end{equation*}
	
	These formulae are obtained by solving the equations $A = 0 = B$ for the coefficient functions $m$ and $n$ respectively, and then evaluating the formulae for $A^+,B^+$ and setting them to zero. Similarly, let $\Phi^-$ be the cubic form of the negative Laplace transformed surface $z^-$ with components $A^-$ and $B^-$. Then we have that both surfaces $S,S^+$ are both quadrics if and only if $A = 0 = B$ and $A^- = 0 = B^-$, which happens if and only if 
	\begin{equation}
		\label{eq-neg_quad_quad}
	\tiny	\begin{aligned}
			&\left(\left(-2 b_x a-2 b a_x-4 a_{x x}+2 c_x+2 b_{x y}\right) \mathsf{k}+\left(-6 b a+6 b_y-12 a_x+6 c\right) \mathsf{k}_x\right) q^2 \\
			& +\left(\left(\left(-4 b a+4b_y-8 a_x+4 c\right) q_x+q_{x x y}\right) \mathsf{k}+3 \mathsf{k}_x q_{x y}\right) q-q_y q_{x x} \mathsf{k}-3 \mathsf{k}_x q_y q_x = 0 \, ,
		\end{aligned}
	\end{equation}
	\begin{equation*}
	\tiny	\begin{aligned}
			& \left(8 b b_x+4 b_{x x}\right) \mathsf{k} q^4+\left(\left(\left(4 b^2+6 b_x\right) q_x+2 b q_{x x}+4 b a^2-6 b_y a-2 b a_y+16 a a_x-4 a c\right.\right. \\
			& \left.\left.-2 b_{y y}+4 a_{x y}+2 c_y-4 r_x\right) \mathsf{k}+\left(2 b a-2 b_y+4 a_x-2 c\right) \mathsf{k}_y+4 aq^3\mathsf{k}\left(-b a+b_y-2 a_x+c\right)\right)  \\
			& +\left(\left(\left(2 b_y-4 a_x\right) q_y-4 q_{x y} a-q_{y x y}\right) \mathsf{k}-q_{x y} \mathsf{k}_y+2 a\mathsf{k} q_{x y}\right) q^2+\left(\left(\left(4 q_y a+q_{y y}\right) q_x\right.\right. \\
			& \left.\left.+3 q_{x y} q_y\right) \mathsf{k}+\left(\mathsf{k}_y-2 a\mathsf{k}\right) q_y q_x\right) q-3 q_y^2 q_x \mathsf{k} = 0 \, .
		\end{aligned}
	\end{equation*}
\end{flushleft}
\end{document}